\documentclass[a4paper, leqno]{amsart}
\usepackage{amssymb, amsmath, amsthm, amsrefs}
\usepackage{hyperref}
\usepackage{mathrsfs}
\usepackage[matha, mathx]{mathabx}
\usepackage[shortlabels]{enumitem}
\usepackage{inputenc}
\usepackage{graphicx}
\usepackage{ dsfont }
\usepackage{xcolor}
\numberwithin{equation}{section}
\usepackage{amssymb, amsmath, amsthm, amsrefs}
\usepackage{hyperref}
\usepackage{mathrsfs}
\usepackage[matha, mathx]{mathabx}
\usepackage[shortlabels]{enumitem}

\usepackage{pict2e,picture}

\makeatletter
\DeclareRobustCommand{\bbDelta}{{\mathpalette\bb@Delta\relax}}
\newcommand{\bb@Delta}[2]{%
  \begingroup
  \sbox\z@{$\m@th#1\Delta$}%
  \dimendef\Dht=6 \dimendef\Dwd=8
  \setlength{\Dwd}{\wd\z@}%
  \setlength{\Dht}{\ht\z@}%
  \begin{picture}(\Dwd,\Dht)
  \put(0,0){$\m@th#1\Delta$}
  \put(.42\Dwd,.7\Dht){\line(10,-26){.25\Dht}}
  \end{picture}%
  \endgroup
}

\usepackage{xcolor}
\numberwithin{equation}{section}
\newtheorem{thm}{Theorem}
\newtheorem{theo}[thm]{Theorem}
\newtheorem{corollary}[thm]{Corollary}
\newtheorem{lemma}[thm]{Lemma}
\newtheorem{proposition}[thm]{Proposition}

\newtheorem{rmk}[thm]{Remark}
\newtheorem{define}[thm]{Definition}
 
\newcommand{\abs}[1]{\left|#1\right|}

\title{Explicit multi-slit Loewner flows and their geometry}
\author[Eleftherios Theodosiadis]{E. K. Theodosiadis}
\address{Department of Mathematics, Stockholm University, 106 91 Stockholm, Sweden}
\email{eleftherios@math.su.se}
\date{\today}


\begin{document} 

\maketitle

\begin{abstract}
In this paper we present explicit solutions to the radial and chordal Loewner PDE and we make an extensive study of their geometry. Specifically, we study multi-slit Loewner flows, driven by the time-dependent point masses $\mu_{t}:=\sum_{j=1}^{n}b_j \delta_{\{\zeta_j e^{iat}\}}$ in the radial case and  $\nu_t:=\sum_{j=1}^{n}b_j\delta_{\{k_j\sqrt{1-t}\}}$ in the chordal case, where all the above parameters are chosen arbitrarily. 

Furthermore, we investigate their close connection to the semigroup theory of holomorphic functions, which also allows us to map the chordal case to the radial one.
\end{abstract}
	\section{Introduction}
	
	Loewner's theory was pioneered in 1923 by Charles Loewner (1893–1968), who studied the properties of continuously evolving families of slit mappings of the unit disc and discovered that such a family satisfies the so-called Loewner PDE. A few years earlier, in 1916, Ludwig Bieberbach proved that the second coefficient of the Taylor series of a univalent (analytic and one-to-one) mapping of the unit disc, satisfies the inequality $|a_2|\le2|a_1|$, while he also conjectured that $|a_n|\le n|a_1|$, for every $n\ge1$. The first theoretical application of Loewner's PDE appeared in the proof of the Bieberbach conjecture for the third coefficient. An affirmative answer to the Bieberbach conjecture was first given by Louis de Branges in 1984, however, a second and simpler proof by Pommerenke and Fitzgerald in 1985 was possible due to Loewner's theory for slit mappings.
	
	A contemporary and more general version of Loewner's PDE was introduced by Pavel Kufarev (1909–1968) in 1947. According to his study, given a continuous and increasing family of simply connected domains, then their corresponding Riemann maps, satisfy a PDE driven by an analytic function of the unit disc with positive real part and conversely, any such PDE produces a family of solutions, whose images form a continuous and increasing family of simply connected domains. Roughly speaking, the driving functions are replaced by driving measures, since analytic functions with positive real part are represented by finite Borel measures. This generalizes Loewner's theory to non-slit mappings. For this reason, nowadays, we often refer to the Loewner-Kufarev PDE.

	\textbf{The radial Loewner PDE}. Considering a simply connected domain $D$ in the Riemann sphere $\hat{\mathbb{C}}$ where we omit at least two points, the Riemann mapping theorem guarantees a conformal map from the unit disk onto $D$. Assuming a point $a\in D$, this map is unique when we require $f(0)=a$ and $f'(0)>0$. 
	
	Let $(D_t)_{t\ge0}$ be a decreasing (resp. increasing), i.e, $D_t \subset D_s $ for $s<t$ (resp. $D_s\subset D_t $), family of simply connected domains in $\hat{\mathbb{C}}$, that is continuous in the sense of Caratheodory's kernel convergence. A decreasing (resp. increasing) Loewner chain is defined to be   the family of the unique Riemann maps described above, $f_t :=f(\cdot,t):\mathbb{D}\rightarrow D_t$. We refer to $f$ as the \textit{Loewner flow}. Assuming that the origin is contained in each $D_t$, then $f_t $ expands in Taylor series as
	 $$f(z,t)=\beta(t)z+\dots, \quad |z|<1,\quad t\ge0 $$
	 where $\beta$ is decreasing (resp. increasing) with time, as follows by the Schwarz lemma. For the rest, we shall only consider decreasing families. By monotonicity, it also follows that $\beta$ is almost everywhere differentiable. For any such family, there exists a family of bounded Borel measures $(\mu_t )_{t\ge0}$ on the unit circle, with mass $\mu_t (\partial\mathbb{D})=-\beta'(t)/\beta(t)$, so that the Loewner-Kufarev equation 
	\begin{equation}
	 \frac{\partial f}{\partial t}(z,t)=-zf'(z,t)\int_{\partial\mathbb{D}}\dfrac{\zeta+z}{\zeta-z}d\mu_t{(\zeta}):=-zf'(z,t)p(z,t)\label{kufarev}
	\end{equation}
	 is satisfied for all $z\in\mathbb{D}$ and for almost all $t\ge0$. In the classical notation, $\beta (t)$ is given by the exponential function.  This can be considered by reparameterizing time, since $\beta$ is monotonic. For a detailed approach, see \cite{dur} and \cite{pom}.

	 Conversely, equation $(\ref{kufarev})$ admits a unique solution for given initial values $f(z,0)$. We are interested in the initial value problem $(\ref{kufarev})$ and when $f(z,0)=z$, thus the starting domain is the unit disk. Then, the solution $f(z,t):\mathbb{D}\rightarrow D_t =\mathbb{D}\setminus K_t$ describes a decreasing Loewner flow. We refer to $(K_t )_{t\ge0}$ as the \textit{growing hulls} of the evolution.
		 
	 In terms of the \textit{transition functions} $\phi_{s,t}(z)=\phi(z;s,t)=f_{s}^{-1}\circ f_t (z)$ for $s<t$ the Loewner-Kufarev ODE is written as
	 \begin{equation}
	 	\dfrac{\partial\phi}{\partial s}(z;s,t)=\phi(z;s,t)p(\phi(z;s,t),s),\ \phi(z;t,t)=z \label{transitionODE}
	 \end{equation}
	  for $z\in\mathbb{D}$ and $0<s<t$. Clearly, it is easier to solve the ODE instead of the corresponding PDE and it will play the basic role in the forthcoming discussion.

	 \textbf{The chordal Loewner PDE.} The situation in the upper half plane is somewhat similar to the radial case. Here, we consider a decreasing family of simply connected domains $(D_t)_{t\ge0}$, such that  $D_t\subset\mathbb{H}$, $K_t:=\mathbb{H}\setminus D_t$ is compact and moreover $\mathbb{H}\setminus K_t$ is also simply connected. Such a family of domains is produced by the so-called \textit{chordal} Loewner equation.
	 
	 Let $D$ be a simply connected domain as above. Then, by Riemann's mapping theorem there exists a conformal map $g$ from $D$ onto $\mathbb{H}$, such that $g(z)-z\rightarrow0$ as $z$ becomes infinite. This property is referred in the literature as the \textit{hydrodynamic condition} of $g$. The chordal Loewner flows are always normalized such that each map of the flow satisfies the hydrodynamic condition.
	 
	 The slit case in the upper half plane is treated as in the radial Loewner's slit case. A recent, detailed proof is given by A. Monaco and P. Gumenyuk in \cite{mon}. See also A. Starnes \cite{star} and M.- N. Technau \cite{tec} for the multiple and infinitely many slits versions respectively. Assume that $\gamma$ is a Jordan curve emanating from $\mathbb{R}$, with parameterization $\gamma=\gamma(t)$, $0\le t<\infty$ and write $K_t:=\gamma([0,t])$. Let $g_t$ be the corresponding Riemann maps described above and let $f_t$ be the inverse mappings. It is then proved that $\lambda(t):=g_t(\gamma(t))$ is a continuous real-valued function of $t$ and for each $T>0$ we have the chordal Loewner ODE 
	 $$\dfrac{\partial g}{\partial t}
	 (z,t)=\dfrac{2}{g(z,t)-\lambda(t)},\quad g(z,0)=0$$
	 for all $z\in\mathbb{H}\setminus\gamma([0,T])$ and $0\le t\le T$. By taking its inverse, the chordal Loewner PDE is written as 
	 $$\dfrac{\partial f}{\partial t}(z,t)=-f'(z,t)\dfrac{2}{z-\lambda(t)}$$
	 in $\mathbb{H}\times[0,\infty)$. The continuous function $\lambda$ is called the driving function of the flow $f(z,t)$.
	 
	 More generally, given a family of probability measures $(\mu_t)_{t\ge0}$ of $\mathbb{R}$ with compact support, we consider the ODE in $\mathbb{H}\times[0,\infty)$
	 $$\dfrac{\partial w}{\partial t}(z,t)=\int_{\mathbb{R}}\dfrac{d\mu_t(x)}{w(z,t)-x}$$
	 with initial value $w(z,0)=0$. Let $T_z$ be the supremum of all $t$, such that the solution to the equation is well defined and $g_t(z)\in\mathbb{H}$ for all $t\le T_z$. Then, there exists a unique solution $g_t$ which is conformal in the domain $D_t:=\{z: T_z>t\}$, satisfying the hydrodynamic condition. Finally, the Loewner flow $f_t=g_t^{-1}$ satisfies the PDE
	 $$\dfrac{\partial f}{\partial t}(z,t)=-f'(z,t)\int_{\mathbb{R}}\dfrac{d\mu_t (x)}{z-x}$$
	 for all $z\in\mathbb{H}$ and $t\ge0$, called the chordal Loewner PDE. We refer to the compact sets $K_t:=\mathbb{H}\setminus D_t$ as the compact hulls generated by the flow. See \cite{lawl} for details.

\textbf{Connection to semigroups.} There is a close connection of Loewner's theory to the semigroup theory for holomorphic self-maps of the unit disc. In fact, due to the Berkson-Porta formula (see \cite{bra}), it turns out that a Loewner flow $f(z,t)$ driven by a time-independent function $p(z,t)=p(z)$, forms a semigroup with fixed point the origin and it is parameterized as 
	 
	 $$f(z,t)=h^{-1}(e^{-t}h(z))$$
	 for all $z\in\mathbb{D}$ and $t\ge0$, where $h$ is a starlike function with respect to zero, called \textit{the Koenigs} function. In \cite{sola1}, A. Sola presents examples of Loewner flows, driven by time-independent densities.
	 
	 To be more precise, a continuous \textit{elliptic semigroup} of holomorphic self-maps of the unit disc, with Denjoy-Wolff point 0, is defined as a family $(\phi_t)_{t\ge0}\subset H(\mathbb{D})$ so that
	 \begin{enumerate}
	     \item $\phi_0=id_{\mathbb{D}}$
	     \item $\phi(z,t+s)=\phi(\phi(z,t),s)$, for all $z\in\mathbb{D}$ and $s,t\ge0$,
	     \item $\phi_t(z)\rightarrow0$, as $t\rightarrow\infty$, for all $z\in\mathbb{D}$,
	     \item $\mathbb{D}\times[0,\infty)\ni(z,t)\mapsto\phi(z,t)$ is continuous in the uniform on compacts topology.
	 \end{enumerate}
For the rest, we shall only refer to continuous semigroups, thus (4) is always assumed. It is proved that there exists some $\lambda\in\mathbb{C}$, with $\text{Re}\lambda\ge0$, so that $\phi'(0)=e^{-\lambda t}$. We call $\lambda$ the \textit{spectral value} of the semigroup. Abbreviating the definition, we might refer to $e^{-\lambda t}$ as the spectral value instead. See proposition 8.1.4 in \cite{bra} for more details on the Denjoy-Wolff theorem and the spectral value.

Now, given a semigroup $(\phi_t)_{t\ge0}$ , then and only then, there exists a unique vector field $G\in H(\mathbb{D})$, such that the PDE
$$\dfrac{\partial\phi}{\partial t}(z,t)=G(\phi(z,t))$$
is satisfied in $\mathbb{D}\times T$, where $T$ is an interval containing $[0,\infty)$. We call $G$ the \textit{infinitesimal generator} of the semigroup. The Berkson-Porta theorem (or formula) states that given a non-constant $G\in H(\mathbb{D})$, then $G$ is the infinitesimal generator of a semigroup if and only if, there exists some $p\in H(\mathbb{D})$ with $\text{Re}(p(z))>0$ in $\mathbb{D}$, so that 
$$G(z)=-zp(z).$$
In the Loewner language, $p$ is the driving function as we previously mentioned.

A semigroup of the upper half plane is defined similarly. Moreover, a non-elliptic semigroup is defined as above, but it has Denjoy-Wolff point in the boundary. Thus, there exists some $\tau\in\partial\mathbb{D}$, such that condition $(3)$ is written as $\phi_t (z)\rightarrow\tau$.

To conclude, we see $h$, the Koenigs function of the semigroup, as the solution to the ODE
$$G(z)=-\lambda\dfrac{h(z)}{h'(z)}$$
in $\mathbb{D}$, where $\lambda$ is the spectral value of the semigroup. Observe by Berkson-Porta's formula and the characterization of spirallike functions (see next section), that by taking the real part of the preceding equation, then $h$ is an $\text{Arg}(\lambda)$-spirallike function of $\mathbb{D}$. For our purposes, it is enough to view the Koenigs function as discussed above. We may also think of it as the function that maps the orbits $\{\phi(z,t)/ t\ge0\}$ onto logarithmic spirals (for elliptic semigroups), or onto half-lines (for non-elliptic semigroups). Of course, the theory for Koenigs function is much deeper. For a general discussion, see chapter 9 in \cite{bra}. 

	\textbf{Brief overview of literature.} Loewner proved that growing curves correspond to continuous driving functions, but the converse is not true. Kufarev presented an example of point masses with the corresponding flow being non-slit. For instance, the driving function $e^{-i\lambda(t)}=k(t)=(e^{-t}+i\sqrt{1-e^{-2t}})^3$ produces the family of domains $D_t$, that are $\mathbb{D}$ minus the part of the disc lying in $\mathbb{D}$ and intersecting $\partial\mathbb{D}$ orthogonally, at the points $k(t)$ and $k(t)^\frac{1}{3}$. In the literature, we find sufficient conditions for a driving function to produce hulls that are curves. In particular, D. Marshall and S. Rohde prove in \cite{MR}, that Lip-$\frac{1}{2}$ driving functions, with sufficiently bounded $\text{Lip}(\frac{1}{2})$-norm, produce quasi-slit domains, while J. Lind, proved in her work in \cite{lind}, $C_0=4$ to be the optimal upper bound for the norm. Although the preceding results refer to single-slit flows, a generalization for multiple-slit flows was done by S. Schleissinger in \cite{slei1}, proving that Lip-$\frac{1}{2}$ driving functions $\lambda_1,\dots,\lambda_n$ produce $n$ disjoint Jordan curves.
	
	Some explicit flows, that are relevant to our work as well, are found in \cite{kad} by P. Kadanoff, B. Nienhuis and W. Kager, where they explicitily solve the PDE for the driving function $k\sqrt{1-t}$ and they describe the geometry of the solutions for the cases $\abs k<4$ and $\abs k>4$. Other cases of slit mappings in the upper half plane are presented in \cite{pro} by D. Prokhorov, A. Zakharov and A. Zherdev. Since the geometry of the slits will be the main topic of this text, we must note the work of C. Wong \cite{won}, according to which, a driving function that is Lipschitz continuous with exponent in $(\frac{1}{2},1]$ produces a curve that grows vertically from the real line.
	
Our aim in this work is to present explicitily given, multi-slit Loewner flows both in the disc and in the upper half plane, by solving their corresponding PDE's. Therefore, the first step in our study, is to write down the Riemann maps produced by particularly chosen driving functions and the second step is to describe their geometry. Finally, we present the close relation of the particular maps to semigroup theory for holomorphic self maps of the unit disc, which allows to visualize a specific class of Loewner PDE's. We outline the structure of the text below.

		\section{Outline of the paper and preliminaries}
	 
	 We begin by presenting a couple of preliminary results about sprirallike domains, that are necessary for the main ideas of the third and fourth section of this text. A \textit{logarithmic spiral of angle} $\psi\in(-\frac{\pi}{2},\frac{\pi}{2})$ in the complex plane is defined as the curve with parameterization $S: w=w_0 \text{exp}(-e^{-i\psi}t), -\infty\le t\le\infty$, for some complex number $w_0\neq0$.

\begin{define}
A simply connected domain $D$, that contains the origin, is said to be $\psi$-spirallike (with respect to zero), if for any point $w_0\in D$, the logarithmic spiral $S: w=w_0\text{exp}(-e^{i\psi}t), 0\le t\le\infty$ is contained in $D$.
\end{define}

\begin{define}
A  univalent function $f\in H(\mathbb{D})$, with $f(0)=0$, is said to be $\psi$- spirallike if it maps the unit disc onto a $\psi$-spirallike domain $D$.
\end{define}
Note that $0$-spirals are straight lines emanating from the origin and expanding to infinity. We refer to $0$-spirallike domains/functions as starlike domains/functions. The following theorem gives an analytic characterization of spirallike mappings. A detailed proof is found in \cite{dur}, paragraph 2.7.
\begin{theo}\label{spiralliketheo}
Let $f\in H(\mathbb{D})$, with $f'(0)\neq0$ and $f(z)=0$ if and only if $z=0$. Then $f$ is $\psi$-spirallike, if and only if 
$$\text{Re}\left(e^{-i\psi}\dfrac{zf'(z)}{f(z)}\right)>0$$
for all $z\in\mathbb{D}$.
\end{theo}

Although spirallike functions usually refer to analytic functions of the disc, we can easily transfer the preceding definition on functions of the upper half plane. 
\begin{define}
A  univalent function $f\in H(\mathbb{H})$, with $f(\beta)=0$ for some $\beta\in\mathbb{H}$, is said to be $\psi$-spirallike (with respect to $\beta$), if it maps the upper half plane onto a $\psi$-spirallike domain $D$.
\end{define}
\begin{proposition} \label{spirallikeinH}
     Let $f\in H(\mathbb{H})$, with $f'(\beta)\neq0$ and $f(z)=0$ if and only if $z=\beta$. Then, $f$ is $\psi$-spirallike, if and only if 
$$\text{Im}\left(e^{-i\psi}\dfrac{(z-\beta)(z-\bar{\beta})f'(z)}{f(z)}\right)>0$$
for all $z\in\mathbb{H}$.   
\end{proposition}
\begin{proof}
 The result follows immediately from the disc case, by applying the Möbius transform $Tz=\frac{z-\beta}{z-\overline{\beta}}$, which maps $\mathbb{H}$ onto $\mathbb{D}$.
\end{proof}
\textbf{Basic results.}	  We start off by setting our configuration in the unit disc. Given some arbitrarily chosen points $\zeta_1 ,\dots,\zeta_n \in\partial\mathbb{D}$,  the weights $b_1 ,\dots,b_n>0$ and the exponent $a\in\mathbb{R}$, the role of the driving measures of the flow will be played by the measures 
	 
$$\mu_t:=\sum_{k=1}^{n}b_k\delta_{\{e^{iat}\zeta_k\}}.$$
	 Note that the flow is not normalized, since we do not demand the measures to be probablities; i.e, $b_1+\dots+b_n=1$. Our first result in paragraph 3.3 is outlined below.
	 
	 \begin{theo}\label{radialtheorem}
Assume the configuration above and consider the radial Loewner-Kufarev PDE in $\mathbb{D}\times[0,\infty)$, 
	 $$\dfrac{\partial f}{\partial t}(z,t)=-f'(z,t)z\sum_{k=1}^{n}b_k\dfrac{e^{iat}\zeta_k +z}{e^{iat}\zeta_k -z}$$
	 with initial value $f(z,0)=z$. Then, the Loewner flow is of the form 
	 $$f_a (z,t)=\phi^{-1}(e^{-(\sum_{k=1}^{n}b_k-ia)t}\phi(e^{-iat}z)),$$
	 where $\phi$ is an $(\text{Arccot}(\frac{a}{\sum_{k=1}^{n}b_k})-\frac{\pi}{2})$-spirallike function of $\mathbb{D}$.
	 
	 Furthermore, for each $k$, with $1\le k\le n$, the trace $\hat{\gamma}_a^{(k)} :=\{f_a (e^{iat}\zeta_k ,t)/\ t\ge0 \}$ is a smooth curve lying in $\mathbb{D}$ that starts perpendicularly from $\zeta_k$, spiralling about the origin. 
\end{theo}
	 Motivated by the work of J. Lind, D. Marshall and S. Rohde, our next objective is to transfer the preceding result to the upper half plane and furthermore, generalize their single-slit construction (see \cite{mar}, section 3). According to this, a curve generated by the driving function $k\sqrt{1-t}$, $t\in[0,1)$ spirals about some point $\beta=\beta(k)\in\mathbb{H}$, or intersects the real line as $t\rightarrow1$ at some point $\rho_0=\rho_0(k)\in\mathbb{R}$ tangentially or non-tangentially, when $\abs k<4$, or $\abs k=4$, or $\abs k>4$ respectively. It turns out that for $n$ many slits, the same geometry can take place, for all curves, while there is one more case, in which the intersection with $\mathbb{R}$ is vertical.

	 We set our configuration by choosing real points $k_1,\dots,k_n$ in increasing order and $b_1,\dots,b_n>0$ some weights, thus the driving measures are written in the form
	 $$\nu_t=\sum_{j=1}^{n}b_j\delta_{\{k_j\sqrt{1-t}\}}$$
	 for $t\in[0,1)$. In our case, where we have multiple real points, the conditions which describe the geometrical behaviour of the curves, are adjusted according to the auxiliary polynomial 
	 $$P(z):=z\prod_{j=1}^{n}(z-k_j)+\sum_{j=1}^{n}4b_j \prod_{i\neq j}(z-k_i).$$
	 It will be obvious later why we choose this polynomial. However, we can already see that $P$ has $n-1$ real roots, say $\lambda_j\in(k_j,k_{j+1})$, because $P(k_j)=(-1)^{n-j}$. Observe that if $n=1$, then $P$ has two complex and conjugate roots for $\abs k<4$, a double real root if $\abs k=4$ and two real and simple roots if $\abs k>4$, in agreement with \cite{mar}. Each of the preceding cases, produce the geometry described in the preceding paragraph. The main result of chapter 4 is formulated in the following theorem.
\begin{thm}\label{chordaltheorem}
Assume the configuration above and consider the chordal Loewner-Kufarev PDE in $\mathbb{H}\times[0,1)$,
	      
	      $$\dfrac{\partial f}{\partial t}(z,t)=-f'(z,t)\sum_{j=1}^{n}\dfrac{2b_j}{z-k_j\sqrt{1-t}}$$
	      with initial value $f(z,0)=z$. Then,
	      \begin{enumerate}
	          \item if $P$ has a complex root $\beta\in\mathbb{H}$, then the Loewner flow is of the form $$f(z,t)=h^{-1}((1-t)^{\alpha e^{-i\psi}}h((1-t)^{-\frac{1}{2}}z))$$
	          where $h$ maps the upper half plane onto the complement of $n$ logarithmic spirals of angle $-\psi$.

	          \item if $P$ has $n+1$ distinct real roots, then the Loewner flow is of the form $$f(z,t)=h^{-1}((1-t)^{\alpha}h((1-t)^{-\frac{1}{2}}z))$$
	          where $h$ is a Schwarz-Cristoffel map of the upper half plane, that maps $\mathbb{H}$ onto $\mathbb{H}$ minus $n$ line segments emanating from the origin.
	          \item if $P$ has a multiple root, either double or triple, then the Loewner flow is of the form $$f(z,t)=h^{-1}(\dfrac{1}{2}\log(1-t)+h((1-t)^{-\frac{1}{2}}z))$$
	          where $h$ is a univalent function map of the upper half plane, that maps $\mathbb{H}$ onto:\\
	          (a) a half plane determined by a translation of $\mathbb{R}$, minus $n$ half-lines parallel to $\mathbb{R}$, if the root is double,\\
	          (b) the complement of $n$ half-lines parallel to $\mathbb{R}$, if the root is triple.  
	      \end{enumerate}
       
            Furthermore, for all $1\le j\le n$, the trajectories of the driving functions $\hat{\gamma}^{(j)}:=\{f(k_j\sqrt{1-t},t)/\ t\in[0,1)\}$, are smooth curves of $\mathbb{H}$ starting at $k_j$, that spiral about the point $\beta$ in case $(1)$, intersect $\mathbb{R}$ at one of the real roots non-tangentially in case $(2)$ and intersect $\mathbb{R}$ at the multiple root tangentially in case $(3a)$ or orthogonally in case $(3b)$.	 
\end{thm}
We can already observe that the flows in the two theorems above have a similar form to the semigroups described in the introduction, therefore it makes sense to underline these properties as remarks in the coming paragraphs.

	\section{Radial Spiroid flows}
	 
	 \subsection{One spiral.} The simplest time-dependent driving function to think of would be the unimodular function $\zeta(t)=e^{it}$ and therefore we consider the initial value  equation 
	 \begin{equation}
	 \dfrac{\partial f}{\partial t}(z,t)=-f'(z,t)z\dfrac{e^{it}+z}{e^{it}-z},\quad f(z,0)=z \label{eit}
	 \end{equation}
	 for all $|z|<1$ and $t\ge0$. In this section, we will describe and visualize the solution to $(\ref{eit})$. It suffices to solve the corresponding Loewner ODE
  \begin{equation}
      \dfrac{dw}{dt}(z,t)=w(z,t)\dfrac{e^{it}+w(z,t)}{e^{it}-w(z,t)}\label{eitODE}
  \end{equation}
 with $w(z,0)=z$. Introducing the transform $v=e^{-it}w$, we solve the equation
	 $$\dfrac{dv}{dt}(z,t)=(1+i)v(z,t)\dfrac{v(z,t)-i}{1-v(z,t)}$$
	for all $|z|<1$ and $t\ge0$. This equation can be solved by separation of variables and by integrating, we deduce that the Loewner flow satisfies the functional equation
	 \begin{equation}
	 	\phi(f(z,t))=e^{(i-1)t}\phi(e^{-it}z)\label{phiof}
	 \end{equation}
	 for all $|z|<1$ and $t\ge0$, with $\phi(z)=z(z-i)^{i-1}$. We observe that $\phi'(0)\neq0$,  $\phi(z)=0$ if and only if $z=0$ and 
	 \begin{equation}\label{spirallike1}
	 Re\left(e^{i\pi/4}\dfrac{z\phi'(z)}{\phi(z)}\right)=Re\left(e^{i\pi/4}i\dfrac{1-z}{i-z}\right)>0
	 \end{equation}
	 for all $|z|<1$. By theorem $\ref{spiralliketheo}$, this implies that $\phi$ is a $(-\frac{\pi}{4})$-spirallike  function of $\mathbb{D}$. In particular, $\phi$ is univalent and as a result we can solve in $(\ref{phiof})$ for $f(z,t)$. Of course, it is challenging to determine the inverse of $\phi$, however we shall write the 'explicit' formula for the flow as 
	 \begin{equation}\label{flow1}
	 f(z,t)=\phi^{-1}(e^{(i-1)t}\phi(e^{-it}z)),\quad \forall(z,t)\in\mathbb{D}\times[0,\infty).
	 \end{equation}
	 
	 Let us, now, study the behaviour of $\phi$. A direct computation gives that 
	 $$\phi(e^{i\theta})=\phi(1)\exp[-e^{-i\frac{\pi}{4}}\sqrt{2}/2(\log(1-\sin(\theta))+\theta)]$$
	 for $0\le\theta<2\pi$, where we have used the identity $1-\sin(\theta)=2\sin(\frac{\pi}{4}-\frac{\theta}{2})\cos(\frac{\pi}{4}-\frac{\theta}{2})$, thus,
  $$\text{Arg}(e^{i\theta}-i)=-\text{Arccot}\left(\dfrac{\cos(\theta)}{1-\sin(\theta)}\right)=-\text{Arccot}\left(\dfrac{\cos(\frac{\pi}{4}-\frac{\theta}{2})}{\sin(\frac{\pi}{4}-\frac{\theta}{2})}\right)=-\frac{\pi}{4}+\frac{\theta}{2}.$$
  Hence $\phi(\partial\mathbb{D})$ is the part of the logarithmic spiral of angle $-\frac{\pi}{4}$, 
	 $$S:\ w=\phi(1)\exp(e^{-i\pi/4}t),\quad -\infty\le t\le\infty,$$
	 joining $\phi(1)$ with $\infty$. We shall denote by $S_+ $ this part of $S$ and by $S_-$ the remaining part $S\setminus S_+ $, that joins $\phi(1)$ with the origin. Therefore, for all $t\ge0$, $e^{(i-1)t}\phi(\partial\mathbb{D})$ denotes the extension of $\phi(\partial\mathbb{D})$ to the point $e^{(i-1)t}\phi(1)$ along the spiral $S_{-}$.
	 
	 \begin{figure}[ht]
	 	\centering
	 	\includegraphics[width=0.4\linewidth]{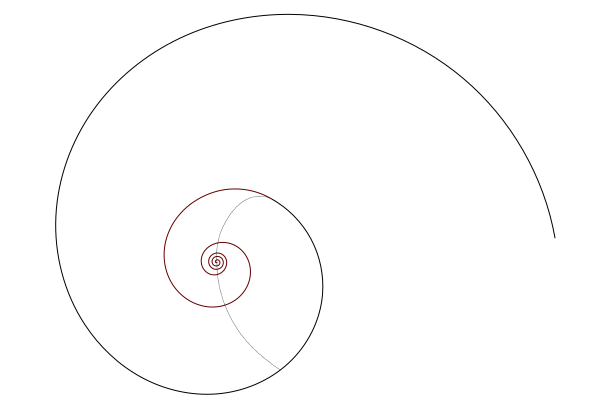}
	 	\caption{Logarithmic spiral of angle $-\pi/4$. The image $\phi(\partial\mathbb{D})$ is denoted by the black part. The dotted line denotes the image $\phi([-1,1])$.}
	 	\label{fig:spiralbr}
	 \end{figure}
	 By $(\ref{flow1})$, applying the inverse of $\phi$ onto $\mathbb{C}\setminus e^{(i-1)t}\phi(\partial\mathbb{D})$ we deduce that $f(\cdot,t)$ maps the unit disc onto $\mathbb{D}$ minus a slit lying in $\mathbb{D}$ with endpoint at $1$, corresponding to the preimage of the segment $S_t :=\{e^{(i-1)x}\phi(1)/ \ 0\le x\le t\}$ of $S_-$. Our purpose is to determine the behaviour of the orbit of the tip point $f(e^{it},t)$, in terms of its winding around the origin. As pointed out above, it is not possible to deduce a closed formula for $f(e^{it},t)$ with time $t$, however, we are able to give a geometric description. 
	 
	 For this reason, consider the image of the diameter $[-1,1]$,
	 $$\phi(x)=\dfrac{x}{\sqrt{x^2 +1}}e^{\mathrm{Arccot}(x)}\exp[i(\log(\sqrt{x^2 +1})+\mathrm{Arccot}(x))]$$
	 so $\text{Arg}(\phi(x))$ decreases from $\text{Arg}(\phi(-1))$ to $-\pi/2$ in $[-1,0)$ and from $\pi/2$ to $\text{Arg}(\phi(1))$ in $(0,1]$. Thus, the imaginary axis acts as the tangent line to $\phi([-1,1])$ at the origin, as seen in figure $\ref{fig:spiralbr}$. Also, because
  $$\phi'(x)=-i\dfrac{1-x}{x^2+1}e^{\text{Arccot}(x)}\exp[i(\log\sqrt{x^2+1}+2\text{Arccot}(x))]$$
  we have that $\text{Arg}(\phi'(x))\rightarrow\frac{1}{2}\log2=S'(0)$, as $x\rightarrow1$ and as a result, $\phi([0,1])$ intersects the spiral $S$, tangentially at the point $\phi(1)$. Note then, that the interior of the curve consisting of $\phi([-1,1])$ and the part of $S_+$ joining its endpoints, is the image of the lower half-disc. Obviously, it contains as many parts of $S_t $, as the number of times $S_t $ intersects $\phi([-1,0])$, while the intersection points accumulate in the origin. Furthermore, since the spiral $S$ winds around the origin infinitely many times, $\phi([0,1])$ intersects $S_t $ as many times as it winds around $0$ and because of that, as $t$ tends to infinity, then the tip point $f(e^{it},t)$ winds around the origin infinitely many times and tends to zero.
	 \begin{figure}[ht]
	 	\centering
	 	\includegraphics[width=0.3\linewidth]{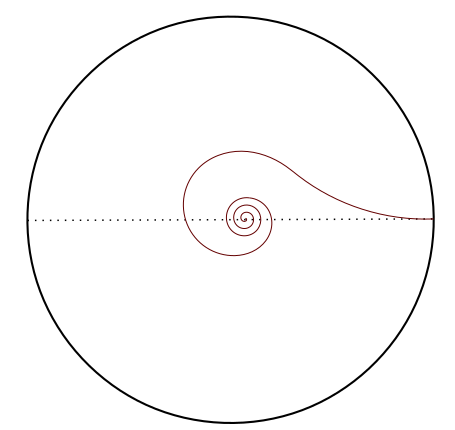}
	 	\caption{Evolution of the slit.}
	 	\label{fig:spiroid1}
	 \end{figure}
	 
	 Let $\hat{\gamma}:=\{f(e^{it},t)/t\ge0)\}$ be the trace of the tip point. By the preceding discussion it is a curve in the unit disk connecting the origin and $1$, spiralling around zero. The reason why $\hat{\gamma}$ is not a logarithmic spiral is simply because it does not hit the positive real radius periodically with $t$. This follows from the fact that the points of intersection of the image of $\phi([-1,1])$ with the spiral $S$, do not lie onto the straight line segment $[S_+(\pi),\phi(1)]$ and therefore $S_-$ does not intersect $\phi([-1,1])$ periodically, as we observe in figure $\ref{fig:spiralbr}$. In fact, there is a "delay" since $[0,\phi(1)]$ lies below $\phi([0,1])$. Asymptotically, however, since the imaginary axis is tangent to $\phi([0,1])$ at the origin, $\hat{\gamma}$ hits the positive radius every $2\pi$ time units. In other words, if $S(t)=\phi(1)e^{(i-1)t}$ parameterizes $S_-$, then $S(t)\in\phi([-1,1])$ for $t=\pi n+t_{n}$, with $t_n>0$ and $t_n \rightarrow0$.
	 
	 Finally, since $\hat{\gamma}$ is the conformal image of $S_-$ under $\phi^{-1}$, we do expect that it will be a distorted image of a spiral around the origin, at least locally. For this we will refer to this curve as \textit{spiroid}. The following proposition, formally verifies that intuition.
	 \begin{proposition}\label{spiralling}
	     If $\hat{\gamma}$ is the trace above, with parameterization $\gamma(t)=r(t)e^{i\Theta(t)}$, then it is spiralling around the origin, i.e, it winds around $0$, $r(t)$ decreases to zero and $\Theta(t)$ increases to infinity.
	 \end{proposition}
	 \begin{proof}
	    By $(\ref{phiof})$ and $(\ref{spirallike1})$, since $\gamma(t)=f(e^{it},t)=\phi^{-1}(e^{(i-1)t}\phi(1))$, differentiating with respect to time we get that for all $t\ge0$,
	    \begin{equation}\label{PDEat1}
	        \gamma'(t)=(r'(t)+r(t)i\Theta'(t))e^{i\Theta(t)}=(i-1)\dfrac{e^{(i-1)t}\phi(1)}{\phi'(\gamma(t))}=\gamma(t)\left(i-\dfrac{1+\gamma(t)}{1-\gamma(t)}\right).
	    \end{equation}
	    Taking the real and imaginary part in the last equation, we deduce the system of the radial and angular parts of Loewner's ODE respectively:
	    \begin{equation}\label{radialPDE}
	    r'(t)=-r(t)\dfrac{1-r^2(t)}{1-2r(t)\cos(\Theta(t))+r^2(t)},\quad r(0)=1
	    \end{equation}
	    and
	    \begin{equation}\label{angularPDE}
	        \Theta'(t)=1-\dfrac{2r(t)\sin(\Theta(t))}
	        {1-2r(t)\cos(\Theta(t))+r^2(t)},\quad \Theta(0)=0.
	    \end{equation}
	    The fact that $r(t)$ is decreasing, follows directly from $(\ref{radialPDE})$. Moreover, in view of equation $(\ref{phiof})$, comparing radial and angular parts we deduce the following equations:
	    $$\dfrac{r(t)}{\sqrt{1-2r(t)\sin(\Theta(t))+r^2(t)}}\exp(-\text{Arg}(r(t)e^{i\Theta(t))}-i))=|\phi(1)|e^{-t}$$
	    and for some $\mu_t\in\mathbb{Z}$, we have that
	    $$\Theta(t)+\dfrac{1}{2}\log(1-2r(t)\sin(\Theta(t))+r^2(t))-\text{Arg}(r(t)e^{i\Theta(t))}-i)=t+\text{Arg}(\phi(1))+2\mu_t\pi.$$
	    Notice that $\mu_t$ is an integer-valued, continuous function of $t$, but for $t=0$, the preceding relation gives that $\mu_0=0$, hence $\mu_t\equiv0$. Moreover, from this equation we deduce that $\Theta(t)\rightarrow\infty$, as $t$ tends to infinity, since $r(t)\rightarrow0$. Combining them together we have the implicit expression for $r$ and $\Theta$
	    $$r^2(t)+1-2r(t)\sin(\Theta(t))=2r(t)e^{2t-\Theta(t)}.$$
	    Now, using $(\ref{angularPDE})$, we want to prove that $\Theta'(t)>0$, for all $t>0$. But by the preceding relation, this means that it suffices to prove that $e^{2t-\Theta(t)}>\cos(\Theta(t))$, and again, it suffices to prove that $e^{2t-\Theta(t)}>1$.
     
	    \begin{figure}[ht]
	 	\centering
	 	\includegraphics[width=0.4\linewidth]{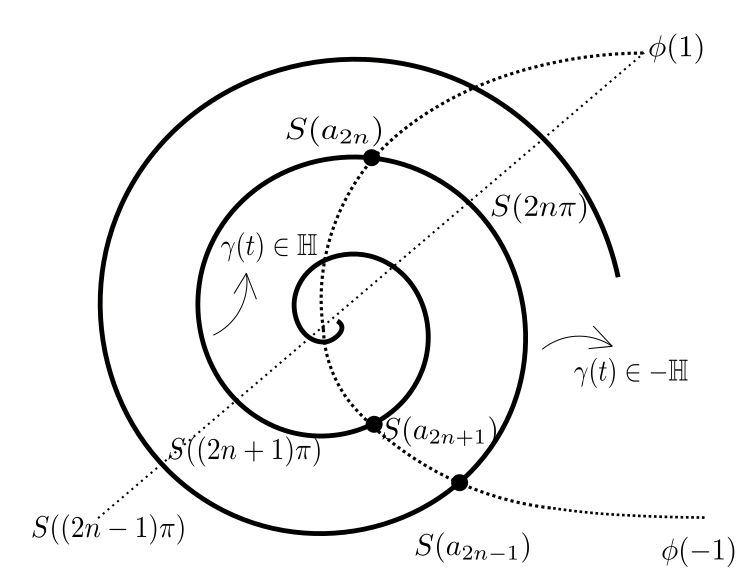}
	 	\caption{Hitting times of $\hat{\gamma}$ with the real diameter. The parts of the spiral between $S(a_{2n-1})$ and $S(a_{2n})$ correspond to the parts of $\hat{\gamma}$ lying in the lower half disc.}
	 	\label{fig:hitingtimes}
	 \end{figure}
  
	    To see this, let the sequence $(a_n)_{n\ge0}$, denote the times when $\hat{\gamma}$ intersects the real diameter. Equivalently, these are the times when $S(t)$ intersects $\phi([-1,1])$, so $(a_n)_{n\ge0}$ is increasing, and we have that $S(a_{2n})\in\phi([0,1])$ and $S(a_{2n+1})\in\phi([-1,0])$.  By the observations we made earlier, we first note that $S([a_{2n},a_{2n+1}])\subset\phi(\mathbb{D}\cap\mathbb{H})$, $S([a_{2n-1},a_{2n}])\subset\phi(-\mathbb{D}\cap\mathbb{H})$ and that
	    $$2n\pi<a_{2n}<(2n+1)\pi<a_{2n+1}$$
     for all $n\ge0$. We also have that $\Theta(a_n)=n\pi$. We see this inductively as following. For $0<t<a_1$, $\gamma(t)$ lies in the upper half disc and hits the interval $[-1,0]$ for the first time when $t=a_1$, since $S([0,a_1])\subset\phi(\mathbb{D}\cap\mathbb{H})$ . Thus $\Theta(a_1)=\pi$. For $a_1<t<a_2$, $\gamma(t)$ lies in the lower half disc and hits the interval $[0,1]$ for the second time when $t=a_2$, since $S([a_{1},a_{2}])\subset\phi(-\mathbb{D}\cap\mathbb{H})$ and thus $\Theta(a_2)=2\pi$. Arguing by induction, therefore, since $\gamma(t)$ lies in one the two half discs in the interval $[a_{n-1},a_n]$ and hence it does not wind about the origin, and since any two consecutive intersections of $\hat{\gamma}$ with $[-1,1]$, are either from the positive radius to the negative radius or vice versa, we then have that $\Theta(a_n)=\pi+\Theta(a_{n-1})=n\pi$.

	    Now, assume that $\gamma(t)$ lies in the upper half disc, thus $a_{2n}<t<a_{2n+1}$, for some $n\ge0$ and $\sin(\Theta(t))>0$. Therefore, by $(\ref{angularPDE})$, $\Theta(t)-t$ is decreasing and so we take that $2t-\Theta(t)>t+a_{2n}-2n\pi>0$. Similarly, in the case where $\gamma(t)$ lies in the lower half disc, then $\Theta(t)-t$ is increasing in an interval of the form $(a_{2n-1},a_{2n})$ and so we get that $2t-\Theta(t)>t+a_{2n}-2n\pi>0$.
	 \end{proof}
	 As a final step to our geometric study, it remains to find the angle in which $\hat{\gamma}$ emanates from the unit circle. By the work of \cite{slei1} and \cite{won}, we directly have that this angle is orthogonal, however, we will present an elementary proof suited for this case. Hence, we need to compute the limit
	 \begin{equation}\label{hitangle}
	     \lim_{t\rightarrow0}\text{Arg}(1-\gamma(t))\quad\text{or}\quad\lim_{t\rightarrow0}\text{Arg}(\gamma'(t)).
	 \end{equation}
	 However, since $\gamma(t)=\phi^{-1}(S(t))$ and $(\phi^{-1})'(\phi(1))=\infty$, the preceding limits require some attention. We have the following proposition.
	 
	 \begin{proposition}\label{orthogonality}
	     If $\hat{\gamma}$ is the trace above, then it hits the unit circle with angle $\frac{\pi}{2}$.
	 \end{proposition}
	 
	 \begin{proof}
	     Let $\gamma(t)=:u(t)+i\upsilon(t)$ written in terms of the real and imaginary part. By the preceding proposition for the radial and angular part of $\gamma(t)$, we deduce that near zero, $u'(t)<0$ and $\upsilon'(t)\ge0$.
	     
	     Differentiating, we have that $\phi'(\gamma(t))\gamma'(t)=S'(t)$ and because of $(\ref{spirallike1})$ we get that 
	     $$(1-\gamma(t))\gamma'(t)=(1+i)\dfrac{\gamma(t)\phi(1)}{\phi(\gamma(t))}e^{(i-1)t}(i-\gamma(t)).$$
	    By taking the limit as $t$ tends to zero we deduce the system
	     \begin{eqnarray}
	         \lim_{t\rightarrow0}\left((1-u(t))u'(t)+\upsilon(t)\upsilon'(t)\right)=-2\label{lim1} \\
	         \lim_{t\rightarrow0}\left((1-u(t))\upsilon'(t)-u'(t)\upsilon(t)\right)=0\label{lim2}
	     \end{eqnarray}
	     Because $u,\upsilon\in C([0,\infty))$ and $u',\upsilon'\in C((0,\infty))$, by $(\ref{lim1})$ we choose $\delta>0$, such that for all $t\in(0,\delta)$,
	     $$|(1-u(t))u'(t)+\upsilon(t)\upsilon'(t)+2|<1$$
	     which implies that
	     $$\left(u(t)-\dfrac{u^2(t)}{2}+\dfrac{\upsilon^2(t)}{2}\right)'<0,\quad \forall t\in(0,\delta).$$
	     But then, since $u(0)=1$ and $\upsilon(0)=0$, we have that $\upsilon^2(t)-(1-u(t))^2 <0$, thus
	     $$0<\dfrac{\upsilon(t)}{1-u(t)}<1$$
	     for all $t\in(0,\delta)$. As a result, $L:=\limsup_{t\rightarrow0}\frac{\upsilon(t)}{1-u(t)}$ exists and we have that $0\le L\le1$. 
	     Again from $(\ref{lim1})$, and because $\upsilon'\ge0$, we get that
	     $$(1-u(t))u'(t)<-1-\upsilon(t)\upsilon'(t)\le-1$$
	     for all $t\in(0,\delta)$. Now, using the preceding bound, the limit in $(\ref{lim2})$ becomes
	     $$\lim_{t\rightarrow0}\left(\dfrac{\upsilon'(t)}{u'(t)}-\dfrac{\upsilon(t)}{1-u(t)}\right)=0.$$
	     Turning to $\limsup$, since $u'<0$ we deduce that $L=\limsup_{t\rightarrow0}\frac{\upsilon'(t)}{u'(t)}\le0$. This shows that the limit $\lim_{t\rightarrow0}\frac{\upsilon(t)}{1-u(t)}$ exists and is equal to zero. As a result, we have that $\text{Arg}(1-\gamma(t))\rightarrow0$ as $t\rightarrow0$ and the result follows.
	 \end{proof}

	 \subsection{A counterexample on convergence of Loewner flows} Of course, the preceding flow could be considered as a special case of the flow $f_a (z,t)$ corresponding to the driving function $\zeta_a (t)=e^{iat}$, for $a=1$. We summarize the above discussion in the following proposition.
	 \begin{proposition}
	Let $a\in\mathbb{R}$ be arbitrary. Then, the Loewner flow $f_a (z,t)$ driven by $\zeta_a (t)=e^{iat}$ is given by the formula
	\begin{equation}
	    f_a (z,t)=\phi_a ^{-1}(e^{(ia-1)t}\phi_a (e^{-iat}z) \label{f_a}
	\end{equation}
	for all $|z|<1$ and $t\ge0$, where $\phi_a \in H(\mathbb{D})$ is  an $(\mathrm{Arccot}(a)-\frac{\pi}{2})$-spirallike function in $\mathbb{D}$ and
	\begin{equation}
	\phi_a (z)=z(z+b)^{-1-b} \label{phi_a}
	\end{equation}
	with $b=b(a)=\frac{1-ia}{1+ia}$. 
	
The trace $\hat{\gamma}_a :=\{f_a (e^{iat},t)/\ t\ge0 \}$ is a smooth curve lying in $\mathbb{D}$ that starts perpendicularly from $1$ and spirals about the origin. Also, $\hat{\gamma}_{-a}$ is a reflection of $\hat{\gamma_a}$ with respect to the real line.
	 \end{proposition}
	 
	 \begin{proof}
	     Consider the function $\phi_a $ as in $(\ref{phi_a})$. Then, the Möbius transform
	     $$\dfrac{z\phi'_a(z) }{\phi_a (z)}=b\dfrac{1-z}{b+z}$$
	    maps the unit disk onto the half plane determined by the line crossing the real axis at 0, with angle $\mathrm{Arccot}(a)-\frac{\pi}{2}$, containing 1. Therefore, we have that
	    $$Re\left(e^{i(\frac{\pi}{2}-\mathrm{Arccot}(a))}\dfrac{z\phi_a '(z)}{\phi_a (z)}\right)>0$$
	    which implies that $\phi_a $ is  an $(\mathrm{Arccot}(a)-\frac{\pi}{2})$-spirallike function of $\mathbb{D}$.
	    In particular $\phi_a $ is univalent, so we define $f_a (z,t)$ by $(\ref{f_a})$. A straightforward differentiation shows that $f_a (z,t)$ solves Loewner's PDE for the driving function $\zeta_a (t)$.
	    
	    Now, taking into account the analysis in section 3.1, the trace $\hat{\gamma}_a $ of the flow is the inverse image of the logarithmic spiral $S_a :w=\phi_a (1)\exp(e^{i(\mathrm{Arccot}(a)-\frac{\pi}{2})}t)$, $-\infty\le t\le0$, under $\phi_a $ and the result follows.
	    
	    Finally, notice that $b(-a)=\overline{b(a)}$. Because the principal branch of $\text{Arg}$ ranges in $[-\pi,\pi)$, thus $\text{Arg}(\bar{w})=-\text{Arg}(w)$, we have the elementary $\overline{(w^a )}=\bar{w}^{\bar{a}}$ and as a result $\overline{\phi_a (z)}=\phi_{-a}(\bar{z})$. Therefore, by conjugating the functional equation $\phi_a (f_a (z,t))=e^{(ia-1)t}\phi_a (e^{-iat}z)$ we take that 
	 $$f_{-a}(z,t)=\overline{f_a (\bar{z},t)}$$
	 for all $a\in\mathbb{R}$.
	 \end{proof}

	 To conclude, we study the convergence of $(f_a )_{a\in\mathbb{R}}$ as $a$ tends to infinity. In general, given a pointwise convergent sequence of driving functions $(p_n )_{n\ge1}$ in $\mathbb{D}\times[0,\infty)$, then the sequence of the corresponding Loewner flows, converges to the Loewner flow corresponding to the limiting driving function. The following counterexample shows that the converse is not necessarily true.
	 
	 \begin{proposition}\label{convergence}
	     The family of Loewner flows $(f_a )_{a\in\mathbb{R}}$ converges to the Loewner flow $h(z,t)=ze^{-t}$. Thus, $f_a (\cdot,t)\rightarrow h(\cdot,t)$ locally uniformly in $\mathbb{D}$, as $a\rightarrow\infty$, for all $t\ge0$.
	 \end{proposition}
	 
	 \begin{proof}
	  As $a$ tends to infinity, we have that $b(a)\rightarrow-1$, hence $\phi_a (z)\rightarrow z$ for all $z\in\mathbb{D}$ and a direct computation also shows that $e^{iat}\phi_a (e^{-iat}z)\rightarrow z$. In fact, the convergence is uniform on the compact subsets of $\mathbb{D}$. To see this, let $|z|\le\rho<1$. Then, $|z+b|\ge\mu:=1-\rho=\text{dist}(b,\partial\mathbb{D}_{\rho})$ and therefore
	 $$|\phi_a (z)|=|z|\exp\left(-\dfrac{2}{a^2 +1}\log|z+b|-\dfrac{2a}{a^2+1}\text{Arg}(z+b)\right)\le\rho\exp\left(\dfrac{2a\pi-2\log\mu}{a^2 +1}\right)$$
	 which implies that $(\phi_a)_{a\in\mathbb{R}}$ is locally uniformly bounded and hence Montel's theorem applies. As a result, by $(\ref{f_a})$ we get that 
	 $$\lim_{a\rightarrow\infty}\phi_a (f_a (z,t))=e^{-t}z$$
	 for all $|z|<1$ and $t\ge0$. 
	 
	 Now, fix any $t\ge0$. Since $f_a (\mathbb{D},t)\subset\mathbb{D}$ for all $a\in\mathbb{R}$, then $(f_a (\cdot,t))_{a\in\mathbb{R}}$ is a normal family, so consider a sequence $(a_n )_{n\ge1}$, such that $f_n :=f_{a_n }(\cdot,t)\rightarrow h $, for some $h\in H(\mathbb{D})$. Note that $h$ will be either univalent or constant by Hurwitz theorem. However, $h'(0)=\lim f_a '(0)=e^{-t}$, hence $h$ is univalent and $|h|<1$. We will prove that $h(z)=ze^{-t}$.
	 
	 Pick an arbitrary $z_0 \in\mathbb{D}$ and choose $\delta>0$, $N\in\mathbb{N}$, so that
	 $$f_n (z_0)\in \overline{D(h(z_0 ),\delta)}\subset\mathbb{D}$$
	 for all $n\ge N$. Then, the locally uniform convergence of $(\phi_a )_a $ implies that
	 
	 $$\lim_{n\rightarrow\infty}(\phi_{a_n } (f_n (z_0))-f_n (z_0))=0$$ 
	 and finally
	 \begin{align*}
	 	|h (z_0)-z_0 e^{-t}|&\le|h(z_0 )-f_n (z_0)|+|\phi_{a_n } (f_n (z_0))-f_n (z_0)|\\&+|\phi_{a_n } (f_n (z_0))-z_0 e^{-t}|\rightarrow0	
	 \end{align*}
	 and the result follows.
	 
	 We, thus, proved that 
	 $$\lim_{a\rightarrow\infty}f_a (z,t)=ze^{-t}$$
	 locally uniformly in $\mathbb{D}$, for each $t\ge0$. 
	  \end{proof}
	 \begin{corollary}
	 There exists a sequence $(f_n )_{n\ge1}$ of Loewner flows that corresponds to some driving functions $(p_n )_{n\ge1}$ and a Loewner flow $f$, such that $f_n\rightarrow f$, but $(p_n )_{n\ge1} $ is not convergent.
	 \end{corollary}	 
	 
	 \begin{proof}
	   We only have to observe that the function $f(z,t)=ze^{-t}$ is the Loewner flow driven by $p(z,t)=1$. The Loewner flows $f_a $ of the preceding proposition converge to $f$, as $a\rightarrow\infty$, but the driving functions $p_a(z,t)=\frac{e^{iat}+z}{e^{iat}-z}$ do not converge with respect to $a$.
	 \end{proof}

	\subsection{The general case} 
	 Let us, now, generalize to the multiple slit case. Let the points $\zeta_1 ,\dots,\zeta_n \in\partial\mathbb{D}$, the weights $b_1 ,\dots,b_n>0$ and the angle $a\in\mathbb{R}$ be arbitrarily chosen and consider  the Loewner ODE
	 \begin{equation}
	 	\dfrac{dw}{dt}(z,t)=w(z,t)\sum_{k=1}^{n}b_k \dfrac{e^{iat}\zeta_k +w(z,t)}{e^{iat}\zeta_k -w(z,t)}
	 \end{equation}
	 for all $|z|<1$ and $t\ge0$, with $w(z,0)=z$. Using the transformation $v=e^{-iat}w$ the equation becomes
	 \begin{equation}
	     \label{ODEu}
	\dfrac{dv}{dt}=v\left(-ia+\dfrac{\sum_{k=1}^{n}b_k (v+\zeta_k )\prod_{j\neq k}(\zeta_j -v)}{\prod_{k=1}^{n}(\zeta_k -v)}\right)=v\dfrac{Q(v)-iaR(v)}{R(v)}
  \end{equation}
	 with $Q(z)=\sum_{k=1}^{n}b_k(\zeta_k +z)\prod_{j\neq k}(\zeta_j -z)$ and $R(z)=\prod_{k=1}^{n}(\zeta_k -z)$. Defining $b:=\sum_{k=1}^{n}b_k$, we have that 
$$Q(z)-iaR(z)=(-1)^{n-1}(b+ia)z^n+\dots+(b-ia)\zeta_1 \dots \zeta_n =:(-1)^{n-1}(b+ia)P_n (z)$$
	 where we set $P_n (z)=\prod_{k=1}^{n}(z-\xi_k )$, for the complex roots $\xi_1 ,\dots,\xi_n $ of the $n$-degree polynomial $Q-iaR$. Now, since $\deg(zP_n (z))=\deg R(z)+1$, we consider coefficients $A, B_1 ,\dots,B_n $ so that 
	 \begin{equation}
	 	\dfrac{R(z)}{zP_n(z)}=\dfrac{A}{z}+\sum_{k=1}^{n}\dfrac{B_k }{z-\xi_k }.\label{coeff}
	 \end{equation}
	 Comparing the coefficients of the polynomials we find that 
	 $$A=\dfrac{R(0)}{P_n (0)}=(-1)^{n-1}\dfrac{(b+ia)R(0)}{Q(0)-iaR(0)}=(-1)^{n-1}\dfrac{1+i\hat{a}}{1-i\hat{a}}=:(-1)^{n-1}\xi_{\hat{a}},$$
	 where $\hat{a}:=a/b$.
	 At this point, let us write $\zeta_k =e^{i\theta_k }$ with $0\le\theta_1 <\dots<\theta_n <2\pi$ and assume for a moment that $\xi_k $ lie on the unit circle, so write $\xi_k =e^{i\rho_k}$ with the angles $\rho_k $ written in increasing order. We then deduce that if $\Theta:=\theta_1 +\dots+\theta_n $ and $P:=\rho_1 +\dots+\rho_n $, then $e^{i(\Theta-P)}=-\xi_{\hat{a}} $. This, of course, would imply that $\Theta-P=\text{Arg}(\xi_{\hat{a}})-\pi+2\mu_0\pi$, for some $\mu_0\in\mathbb{Z}$. Similarly to the prior comparison we take that
	 
	 \begin{equation}\label{a_k}
	  	B_k =\dfrac{R(\xi_k)}{\xi_k \prod_{j\neq k}(\xi_k -\xi_j )}=(-1)^{n-1}\dfrac{\zeta_k -\xi_k }{\xi_k }\prod_{j\neq k}\dfrac{\zeta_j -\xi_k }{\xi_j -\xi_k }
	  	\end{equation}
	  	
	 	\begin{eqnarray*}
	 	&=&2(-1)^{n-1}ie^{i\frac{\Theta-P}{2}}\sin(\dfrac{\theta_k -\rho_k }{2})\prod_{j\neq k}\dfrac{\sin(\frac{\theta_j -\rho_k }{2})}{\sin(\frac{\rho_j -\rho_k }{2})}\\
	 	&=:&2(-1)^{n-1}e^{i\frac{\text{Arg}\xi_{\hat{a}}}{2}}(-1)^{\mu_0}\tilde{\alpha}_k .
	 \end{eqnarray*}
	 In addition, by comparing the coefficients of $z^n$ and by setting $a_k:=(-1)^{\mu_0}\tilde{a}_k$, equation $(\ref{coeff})$ gives us that 
	 \begin{equation}\label{sumak}
	 	\sum_{k=1}^{n}\alpha_k =-\cos(\dfrac{\text{Arg}\xi_{\hat{a}} }{2}).
	 \end{equation}
	 The following lemma not only shows that the $\xi_k $'s lie on the unit circle, but gives us their relative positions in comparison to the $\zeta_k $'s as well.
  
	 \begin{lemma}\label{parameters}
	 	Given the parameters above, the following hold:\\
	 	1. The roots $\xi_1 ,\dots,\xi_n $ of $P_n$ are distinct points of $\partial\mathbb{D}\setminus\{\zeta_1 ,\dots,\zeta_n \}$.\\
	 	2. All coefficients $\alpha_k $ are negative, satisfying $(\ref{sumak})$.
	 \end{lemma}
	 \begin{proof}
	 	1. We have that
	 	\begin{eqnarray*}
	 		Q(z)-iaR(z)&=&\sum_{k=1}^{n}b_k(\zeta_k +z)\prod_{j\neq k}^{n}(\zeta_j -z)-ia\prod_{k=1}^{n}(\zeta_k -z)\\
	 		&=&\prod_{k=1}^{n}(\zeta_k -z)(\sum_{k=1}^{n}b_k\dfrac{\zeta_k +z}{\zeta_k -z}-ia)\\
	 		&=&(1+i\hat{a})R(z)\sum_{k=1}^{n}b_k\dfrac{z+\zeta_k \bar{\xi}_{\hat{a}}}{\zeta_k -z}
	 	\end{eqnarray*}
	 	and notice that $Q(\zeta_k )-iaR(\zeta_k )=2b_k \zeta_k \neq0$. As a result, the points $\xi_k $ are zeros of the sum $g(z)=:\sum_{k=1}^{n}b_k\frac{z+\zeta_k \bar{\xi}_{\hat{a}}}{\zeta_k -z}=:\sum_{k=1}^{n}b_k T_{\hat{a},k}(z)$ and they do not belong to the set $\{\zeta_1 ,\dots,\zeta_n \}$. Now, each of the summands $T_{\hat{a},k}$ is a Möbius transform that maps the unit disk onto the half plane determined by the line $L_{\hat{a}} =\{ie^{-i\frac{1}{2}\text{Arg}\xi_{\hat{a}} } x/\ x\in\mathbb{R}\}$, containing $\bar{\xi}_{\hat{a}} \notin L_{\hat{a}} $. This is a convex domain independent of $k$.
	 	
	 	Therefore, if we assume that there is some $\xi_i $ that belongs to $\mathbb{D}$, then the point $T_{\hat{a},k}(\xi_i )$ will lie in the half plane described above, for all $k$. But then, $0=\frac{1}{b}g(\xi_i )$ is a convex sum of these points, thus it cannot lie in $L_{\hat{a}} \ni0$. Similarly if we assume that $\xi_i \in\overline{\mathbb{D}}^c $. This contradiction shows that the $\xi_k $'s are points of the unit circle.
	 	
	 	To prove that they are distinct, it suffices to show that they are simple roots of $g$. In particular, we have that 
	 	
	 	$$g'(\xi_i )=\sum_{k=1}^{n}b_k \zeta_k \dfrac{1+\bar{\xi}_{\hat{a}}}{(\zeta_k -\xi_i )^2}=-\dfrac{1+\bar{\xi}_{\hat{a}}}{\xi_i }\sum_{k=1}^{n}\dfrac{b_k }{4\sin(\frac{\theta_k -\rho_i}{2})^2 }\neq0$$
	 	and the first part follows.
	 	
	 	2. It will suffice to prove a stronger fact for the positions of the $\xi_k $'s. In particular, we will show that $\xi_k \in(\zeta_k ,\zeta_{k+1})$. From the first part we have that 
	 	$$g'(e^{i\theta})=-\dfrac{1+\bar{\xi}_{\hat{a}}}{e^{i\theta}}\sum_{k=1}^{n}\dfrac{b_k }{4\sin(\frac{\theta_k -\theta}{2})^2 }$$
	 	and so we observe that the function $\tilde{g}(\theta)=-\frac{i}{1+\bar{\xi}_{\hat{a}}}g(e^{i\theta})$ is real-valued with 
	 	
	 	$$\partial_\theta \tilde{g}(\theta)=-\sum_{k=1}^{n}\dfrac{b_k }{4\sin(\frac{\theta_k -\theta}{2})^2 }<0.$$
	 	But since $\tilde{g}(\theta_k^-)=-\infty=-\tilde{g}(\theta_k^+)$, we deduce by monotonicity, that $\tilde{g}$ is zero at exactly $n$ points $\rho_1\dots\rho_n$ and it is either $\rho_k \in(\theta_k ,\theta_{k+1})$ for all $k$, when $\tilde{g}(0)<0$, or $\rho_{k+1}\in(\theta_k ,\theta_{k+1})$ for all $k$, when $\tilde{g}(0)>0$. Hence, we proved that the following cases can occur:
	 	
	 	\begin{equation}\label{angledistribution}0\le\theta_1 <\rho_1 <\theta_2 <\dots<\theta_n <\rho_n \le2\pi\end{equation}
	 	\begin{equation}\label{angledistribution1}0\le\rho_1 <\theta_1 <\rho_2<\theta_2 <\dots<\rho_n<\theta_n  \le2\pi\end{equation}
	 	Such an ordering for the angles will give us that for any $k$, the coefficient

	 	$$\tilde{\alpha}_k =\sin(\dfrac{\theta_k -\rho_k }{2})\prod_{j=1}^{k-1}\dfrac{\sin(\frac{\theta_j -\rho_k }{2})}{\sin(\frac{\rho_j -\rho_k }{2})}\prod_{j=k+1}^{n}\dfrac{\sin(\frac{\theta_j -\rho_k }{2})}{\sin(\frac{\rho_j -\rho_k }{2})}$$
	 	is negative in the first case and positive in the second one. However, because the right-hand part of $(\ref{sumak})$ is always negative, since $\text{Arg}(\xi_{\hat{a}})/2\in(-\frac{\pi}{2},\frac{\pi}{2})$, then $\mu_0\in2\mathbb{Z}$ if $(\ref{angledistribution})$ holds or $\mu_0\in2\mathbb{Z}+1$ if $(\ref{angledistribution1})$ holds. As $a_k=(-1)^{\mu_0}\tilde{a}_k$ by definition, the result follows.
	 	
	 		\begin{figure}[ht]
	 		\centering
	 		\includegraphics[width=0.7\linewidth]{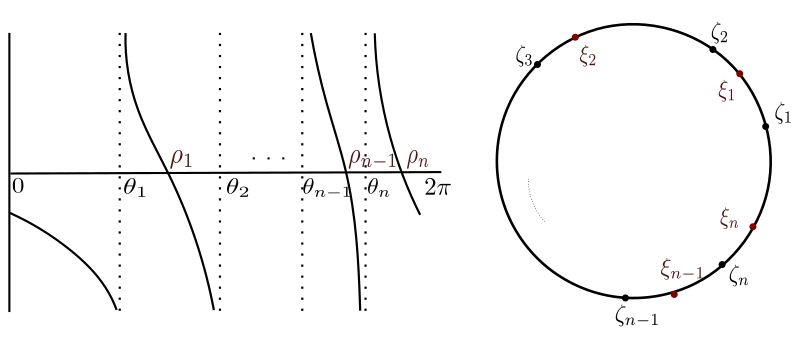}
	 		\caption{A graph of $\tilde{g}$ and the distribution of the roots in the unit circle.}
	 		\label{fig:distribution}
	 	\end{figure}
	 \end{proof}
	 
	 We, now, return to the ODE
	 
	 $$\dfrac{R(v)}{vP_n (v)}dv=(-1)^{n-1}(b+ia)dt$$
	 which becomes, using partial fraction decomposition and the parameters introduced above, 
	 $$\left(\dfrac{1}{v}+\sum_{k=1}^{n}\dfrac{2\alpha_k e^{-i\frac{\text{Arg}\xi_{\hat{a}}}{2}}}{v-\xi_k }\right)dv=(b-ia)dt.$$
	 It is, now, an easy consequence by integration that $\phi(e^{-iat}w(z,t))=e^{(b-ia)t}\phi(z)$,
	 where 
	 \begin{equation}
	 \phi(z)=z\left(\prod_{k=1}^{n}(z-\xi_k )^{2\alpha_k }\right)^{e^{-i\frac{\text{Arg}\xi_{\hat{a}}}{2}}}
	 \end{equation}
	 for all $|z|<1$ and $t\ge0$. Note that $\phi$ depends on $a$, but we shall not write $\phi=\phi_a $ unless necessary. By the preceding lemma, we deduce that the points $\xi_k $ are mapped to infinity and moreover the derivative is given by
	 \begin{equation}\label{phi'}
	 	\phi'(z)=(-1)^{n-1}e^{-i\text{Arg}\xi_{\hat{a}}}\phi(z)\dfrac{R(z)}{zP_n (z)}
	 \end{equation}
	 thus it is zero at the points $\zeta_1 ,\dots,\zeta_n $. Using this formula and recalling that $\text{Arg}\xi_{\hat{a}} =\pi-2\mathrm{Arccot}(\hat{a})$ we have that 
	 $$e^{i(\frac{\pi}{2}-\mathrm{Arccot}(\hat{a}))}\dfrac{z\phi'(z)}{\phi(z)}=e^{i(\frac{\pi}{2}-\mathrm{Arccot}(\hat{a}))}+\sum_{k=1}^{n}\dfrac{2\alpha_k z}{z-\xi_k }$$
	 for all $|z|<1$. Taking into consideration that each of the above summands is a Möbius transform, mapping the unit disk onto the right half plane determined by the perpendicular line at $\alpha_k $, then by $(\ref{sumak})$
	 
	 $$\text{Re}\left(e^{i(\frac{\pi}{2}-\mathrm{Arccot}(\hat{a}))}\dfrac{z\phi'(z)}{\phi(z)}\right)>\text{Re}(e^{i\frac{\text{Arg}\xi_{\hat{a}} }{2}})+\sum_{k=1}^{n}\alpha_k =0$$
	 and because $\phi$ is zero only at the origin and $\phi'(0)\neq0$, we deduce that $\phi$ is a $-\frac{1}{2}\text{Arg}\xi_{\hat{a}} $-spirallike function of $\mathbb{D}$. In particular, it is univalent and therefore the Loewner flow $f(z,t)$ is explicitly written as 
	 
	 \begin{equation}
	 	f(z,t)=\phi^{-1}(e^{-(b-ia)t}\phi(e^{-iat}z))\label{flowa}
	 \end{equation}
	 for all $|z|<1$ and $t\ge0$.
	 
	 Next, we study the geometry of the slits produced by the flow. We wish to answer to the question: What does the image $f (\mathbb{D},t)$ look like? A straightforward calculation for the boundary values of $\phi$, gives us that
	 $$\phi(e^{i\theta})=C\exp\left(-e^{-i\frac{\text{Arg}\xi_{\hat{a}} }{2}}\Theta(\theta)+ie^{-i\frac{\text{Arg}\xi_{\hat{a}} }{2}}\sum_{k=1}^{n}2\alpha_k \text{arg}(\sin(\dfrac{\theta-\rho_k }{2}))\right)$$
	 where $C=\exp(e^{-i\frac{\text{Arg}\xi_{\hat{a}}}{2}}i\sum_{k=1}^{n}\alpha_k \rho_k-2e^{-i\frac{\text{Arg}\xi_{\hat{a}} }{2}}\cos(\frac{\text{Arg}\xi_{\hat{a}}}{2})\log(2i))$ is a constant and
	 $$\Theta(\theta)=\theta\sin(\frac{\text{Arg}\xi_{\hat{a}} }{2})-\sum_{k=1}^{n}2\alpha_k \log|\sin(\frac{\theta-\rho_k }{2})|.$$ 
	 For each $k$, we then have the formula
	 
	 $$\phi(e^{i\theta})=\phi(e^{i\theta_k })\exp\left(-e^{-i\frac{\text{Arg}\xi_{\hat{a}}}{2}}(\Theta(\theta)-\Theta(\theta_k ))\right)$$
	 for all $\theta\in(\rho_{k-1} ,\rho_k)$. Notice that $\Theta'(\theta_k )=0$ because $\phi'(e^{i\theta_k })=0$ and we have that
	 
	 $$\Theta'(\theta)=\sin(\dfrac{\text{Arg}\xi_{\hat{a}} }{2})-\sum_{k=1}^{n}2\alpha_k \cot(\dfrac{\theta-\rho_k }{2})\quad \text{and} \quad \Theta''(\theta)=\sum_{k=1}^{n}\dfrac{2\alpha_k }{\sin(\frac{\theta-\rho_k }{2})^2 }<0.$$

	 We can, therefore, see that the image of $(\rho_{k-1}, \theta_k )$ under $\phi$ is a logarithmic spiral of angle $\frac{1}{2}\text{Arg}\xi_{\hat{a}} $ joining infinity with $\phi(\zeta_k )$ and similarly the image of $(\theta_k ,\rho_k )$ joins $\phi(\zeta_k )$ with infinity through the same spiral. In fact, the above analysis yields that the image of the unit disk under $\phi$ is the complement of $n$ logarithmic spirals $\tilde{S_k }: \phi(\zeta_k )\exp(e^{-i\frac{\text{Arg}\xi_{\hat{a}} }{2}}t)$, $t\ge0$, of angle $-\frac{\text{Arg}\xi_{\hat{a}}}{2}$ joining infinity with the tip points $\phi(\zeta_k )$. These spirals are parts of the total spiral paths $S_k : \phi(\zeta_k )\exp(e^{-i\frac{\text{Arg}\xi_{\hat{a}} }{2}}t)$, $t\in\mathbb{R}$, from infinity to the origin.

 \begin{figure}[ht]
	 	\centering
	 	\includegraphics[width=0.8\linewidth]{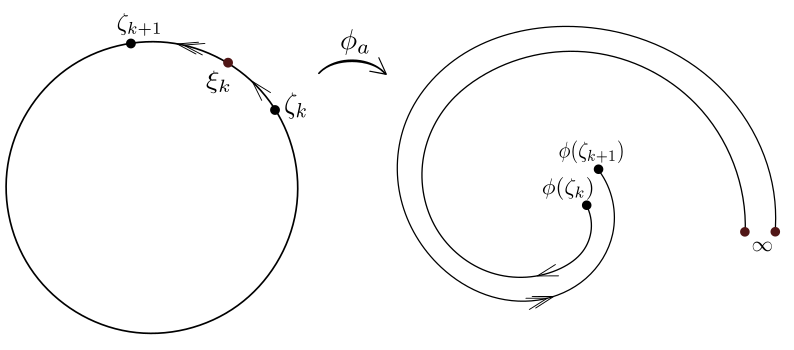}
	 	\caption{The image of the arc $(\zeta_k ,\zeta_{k+1})$ under $\phi$.}
	 	\label{fig:phimap}
	 \end{figure}
  
	 Our next step is to extend the spirals along the spiral paths connecting the origin with the tip points $\phi(\zeta_k )$, by looking at the multiplying factor in equation $(\ref{flowa})$. Indeed, we have that 
	 $$e^{-(b-ia)t}\phi(e^{i\theta})=\phi(e^{i\theta_k })\exp\left(e^{-i\frac{\text{Arg}\xi_{\hat{a}} }{2}}(-bt\sqrt{\hat{a}^2 +1}+\Theta(\theta_k )-\Theta(\theta))\right)$$
	 and this means that the function $z\mapsto e^{-(b-ia)t}z$ maps the spirals $\tilde{S_k }$ onto their extensions, along the paths $S_k $, from the points $\phi(\zeta_k )$ to $e^{-(b-ia)t}\phi(\zeta_k )$. Finally, applying the inverse $\phi^{-1}$, then the Loewner flow maps the unit disk onto $\mathbb{D}$ onto $n$ slits lying in $\mathbb{D}$ except for the endpoint $\zeta_k\in\partial\mathbb{D}$.

	 It then follows that in order to keep track of the total trajectory of the tip points $f(e^{iat}\zeta_k ,t)$, we only need to look at the preimages $\phi^{-1}(\hat{S_k })$, of the spirals $\hat{S_k }\subset S_k $ connecting the origin to the points $\phi(\zeta_k )$ respectively. Putting everything together, we have proved the following result.

	 \begin{theo}\label{Radialflow}
	 Let the points $\zeta_1 ,\dots,\zeta_n \in\partial\mathbb{D}$, the weights $b_1 ,\dots,b_n>0$ and the angle $a\in\mathbb{R}$ be arbitrarily chosen. Consider the points $\xi_1,\dots,\xi_n\in\partial\mathbb{D}$ and the exponents $a_1,\dots,a_n$ by Lemma \ref{parameters}. Then, the Loewner-Kufarev PDE in $\mathbb{D}\times[0,\infty)$, 
	 $$\dfrac{\partial f}{\partial t}(z,t)=-f'(z,t)z\sum_{k=1}^{n}b_k\dfrac{e^{iat}\zeta_k +z}{e^{iat}\zeta_k -z}$$
	 with $f(z,0)=z$, admits the unique solution $f(z,t)=\phi^{-1}(e^{-(b-ia)t}\phi(e^{-iat}z))$, where $\phi=\phi_{a,b_1,\dots b_n}$ is an $(\text{Arccot}(\frac{a}{\sum_{k=1}^{n}b_k})-\frac{\pi}{2})$ -spirallike function of $\mathbb{D}$, given by the formula
	 $$\phi(z)=z\left(\prod_{k=1}^{n}(z-\xi_k )^{2\alpha_k }\right)^{e^{i(\text{Arccot}(\frac{a}{\sum_{k=1}^{n}b_k})-\frac{\pi}{2})}}.$$
	 
	 For each $k$, with $1\le k\le n$, the trace $\hat{\gamma}_k :=\{f(e^{iat}\zeta_k ,t)/\ t\ge0 \}$ is a smooth curve lying in $\mathbb{D}$ that starts perpendicularly from $\zeta_k$, spiralling about the origin when $a\neq0$. 
	 \end{theo}
	 \begin{proof}
Consider the $k$-th trace $\gamma_k(t):=r_k(t)e^{i\Theta_k(t)}=f(\zeta_ke^{iat},t)$. Differentiating with respect to time and using relation $(\ref{phi'})$, then taking the real and imaginary part in $(\ref{ODEu})$, we get the radial and angular part of Loewner's equation, thus, for all $t\ge0$
$$r_k'(t)=-r_k(t)\sum_{j=1}^{n}b_j\dfrac{1-r_k^2(t)}{1-2r_k(t)\cos(\Theta_k(t)-\theta_j)+r_k^2(t)}$$
and
$$\Theta_k'(t)=a-\sum_{j=1}^{n}b_j\dfrac{2r_k(t)\sin(\Theta_k(t)-\theta_j)}{1-2r_k(t)\cos(\Theta_k(t)-\theta_j)+r_k^2(t)}$$
respectively. By the first one we deduce that $r'_k$ is negative and because $\gamma_k(t)\rightarrow0$, we have that $r_k(t)$ tends to $0$, decreasingly as $t\rightarrow\infty$. Therefore, for a sufficiently large $T>0$, we have by the second equation that for all $t>T$, the sign of $\Theta_k'(t)$ is the same as $a$. In addition, comparing the angular parts of the equality $\phi(\gamma_k(t))=e^{-(b-ia)t}\phi(\zeta_k)$ as in proposition $\ref{spiralling}$, we get that $\Theta_k(t)\rightarrow\infty$, as $t\rightarrow\infty$. Therefore, we have that $\Theta_k$ tends to infinity, increasingly if $a>0$ and decreasingly if $a<0$, after time $T$.
  
	 To conclude, we only need to verify that the curves $\hat{\gamma}_k$ intersect $\partial\mathbb{D}$ orthogonally. So, for any $k$, consider the rotated curve $\tilde{\gamma}_k=e^{-i\theta_k}\hat{\gamma}_k$, which starts from $1$. We then have that $\phi(\zeta_k \tilde{\gamma}_k(t))=e^{(ia-b)t}\phi(\zeta_k)$. Following the proof of proposition $\ref{orthogonality}$, differentiating with time, by $(\ref{phi'})$ we have that
	 $$(-1)^{n-1}e^{-i\text{Arg}\xi_{\hat{a}}}\phi(\zeta_k\tilde{\gamma}_k(t))R(\zeta_k\tilde{\gamma}_k(t))\tilde{\gamma}_k'(t)=(ia-b)e^{(ia-b)t}\phi(\zeta_k)P_n(\zeta_k\tilde{\gamma}_k(t))\tilde{\gamma}_k(t).$$
	 By the definition of the polynomials $R(z)$ and $P_n(z)$, the last relation becomes
	 $$(-1)^{n-1}\prod_{j=1}^{n}(\zeta_j-\zeta_k\tilde{\gamma}_k(t))\tilde{\gamma}_k'(t)=(ia-b)e^{(ia-b)t}\frac{\phi(\zeta_k)e^{i\text{Arg}\xi_{\hat{a}}}\tilde{\gamma}_k(t)}{\phi(\zeta_k\tilde{\gamma}_k(t))}P_n(\zeta_k\tilde{\gamma}_k(t))$$
	 or
	 $$\zeta_k(1-\tilde{\gamma}_k(t))\tilde{\gamma}_k'(t)=(ia-b)e^{(ia-b)t}\frac{\phi(\zeta_k)e^{i\text{Arg}\xi_{\hat{a}}}\tilde{\gamma}_k(t)}{\phi(\zeta_k\tilde{\gamma}_k(t))}(\zeta_k\tilde{\gamma}_k(t)-\xi_k)\prod_{j\neq k}\frac{\zeta_k\tilde{\gamma}_k(t)-\xi_j}{\zeta_k\tilde{\gamma}_k(t)-\zeta_j}.$$
	 Letting $t\rightarrow0$, relations $(\ref{lim1})$ and$(\ref{lim2})$ are written as
	 $$\lim_{t\rightarrow0}(1-\tilde{\gamma}_k(t))\tilde{\gamma}_k'(t)=-2b(-1)^{\mu_0}\sqrt{\hat{a}^2+1}\sin(\frac{\rho_k-\theta_k}{2})\prod_{j\neq k}\frac{\sin(\frac{\rho_j-\theta_k}{2})}{\sin(\frac{\theta_j-\theta_k}{2})},$$
    	 where the integer $\mu_0$ is defined by $(\ref{a_k})$. By lemma $\ref{parameters}$ the right-hand part is negative. In addition, consider $c_k(x):=\phi(\zeta_kx)$ the image of the radius with the endpoint $\zeta_k$. Then, we deduce by $(\ref{phi'})$ that $\text{arg}(c_k'(x))\rightarrow\text{arg}(\phi(\zeta_k))-\text{Arg}(\xi_{\hat{a}})$, as $x\rightarrow1$. This implies that the image of the this radius intersects the spiral $\tilde{S_k }: \phi(\zeta_k )\exp(e^{-i\frac{\text{Arg}\xi_{\hat{a}}}{2}}t)$ at its tip point tangentially and therefore, for $t$ close to zero, $\tilde{\gamma}_k(t):=u_k(t)+i\upsilon_k(t)$ does not oscillate between the upper and the lower half discs. Henceforth, we either have that $\upsilon_k(t), \upsilon_k'(t)>0$ ($\tilde{\gamma}_k$ firstly enters the upper half disc) or $\upsilon_k(t), \upsilon_k'(t)<0$  ($\tilde{\gamma}_k$ firstly enters the lower half disc). As a result, the similar argument to proposition $\ref{orthogonality}$ can be applied.
	 	 \begin{figure}[ht]
	 	\centering
	 	\includegraphics[width=0.4\linewidth]{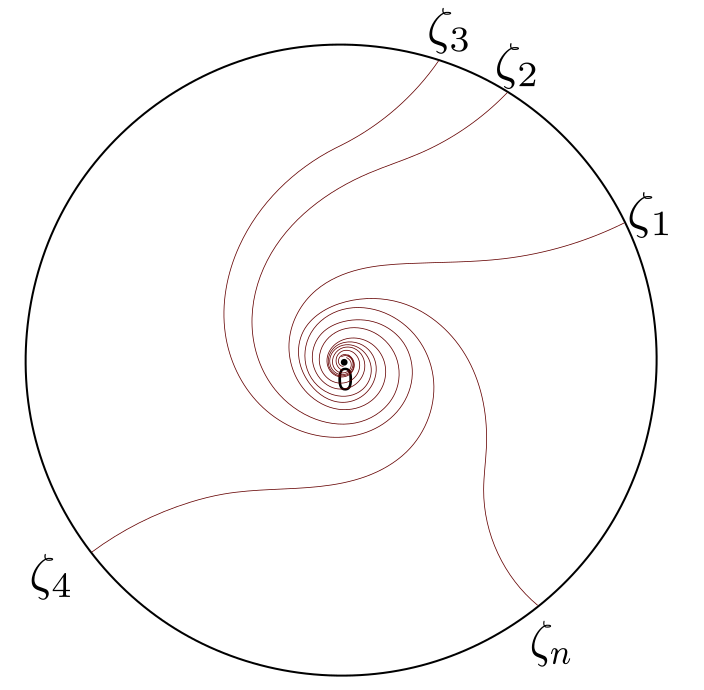}
	 	\caption{Evolution of the $n$ spiroid slits.}
	 	\label{fig:spiraoeids}
	 \end{figure}
	 
	 \end{proof}

	 \subsection{A semigroup property.} It is known that a Loewner flow $f(z,t)$ driven by a time-independent function $p(z,t)=p(z)$, forms a semigroup with fixed point the origin and it is parametrized as 
	 
	 $$f(z,t)=h^{-1}(e^{-t}h(z))$$
	 for all $z\in\mathbb{D}$ and $t\ge0$, where $h$ is a starlike function with respect to the origin, called the Koenigs function.
	 
	 In our case, however, we have that for all $z\in\mathbb{D}$ and $t\ge0$
	 \begin{equation}
	 	f(z,t)=\phi^{-1}(e^{-(b-ia)t}\phi(e^{-iat}z))\label{spiroid}
	 \end{equation}
	 where $\phi$ is a spirallike function of angle $-(\frac{\pi}{2}-\mathrm{Arccot}(\frac{a}{b})$, with respect to the origin and we directly observe that $f$ satisfies the functional equation
	 
	 \begin{equation}
	 	f(e^{ia(t+s)}z,t+s)=f(e^{iat}f(e^{ias}z,s),t)\label{semigroup}
	 \end{equation}
	 which implies that $f(e^{iat}z,t)$ is a semigroup with Denjoy-Wolff point the origin and spectral value $e^{-(b-ia)t}$. 
	 
	 From this point of view, the Koenigs function is the spirallike function $\phi$ and the infinitesimal generator of the semigroup, which in terms of the Loewner equation is just the driving function times $z$, is written as 
	 \begin{equation}
	 	zp(z)=iaz+(b-ia)\dfrac{\phi(z)}{\phi'(z)}\label{phiODE}
	 \end{equation}
	 as equation $(\ref{phi'})$ shows. We observe that although the driving function is time dependent, the dependence is "weak" in the sense that $p(z,t)=p(e^{-iat}z)$. Notice that for $a=0$, we arrive at the case of the time independent driving function, $f(z,t)$ is a semigroup with Denjoy-Wolff point the origin and $\phi$, the Koenigs function, is starlike.
	 
	 Reasoning conversely, it makes sense to consider a continuous family of functions $(f(\cdot,t))_{t\ge0}$ with $f(0,t)=0$ and $f'(0,t)=e^{-bt}$, satisfying equation $(\ref{semigroup})$, for a given $a\in\mathbb{R}$.  Then, there exists a spirallike function $\phi\in H(\mathbb{D})$ of angle $-(\frac{\pi}{2}-\mathrm{Arccot}(\frac{a}{b}))$ so that
	 $$f(e^{iat}z,t)=\phi^{-1}(e^{-(b-ia)t}\phi(z))$$
	 for all $z\in\mathbb{D}$ and $t\ge0$ and as a result $f$ is a Loewner flow which has the same form as the spiroid flow, with driving function $p(z,t)=p(e^{-iat}z)$, where $p$ is again given by $(\ref{phiODE})$.
	 
	 \textbf{Spiroid flows.}  Let us formualate the preceding discussion in the following proposition.
	 
	 \begin{proposition}\label{loewner-semigroup}
	         Let $p\in H(\mathbb{D})$, with positive real part, $p(0)=b$ and let $p(z,t)=p(e^{-iat}z)$ be the driving function for Loewner's PDE. Then, there exists a spirallike function $\phi$, of angle $-(\frac{\pi}{2}-\mathrm{Arccot}(\frac{a}{b}))$, so that the Loewner flow $f(z,t)$ is given by $(\ref{spiroid})$.
	 \end{proposition}
	 \begin{proof}
	  We consider $\phi$ as the solution to ODE $(\ref{phiODE})$, such that $\phi(0)=0$. By the hypothesis, taking the real part $\phi$ is spirallike and the result folllows.
	 \end{proof}
Due to the semigroup theory, the spiroids of the preceding section follow a simple geometric structure. We borrow the definition of self similarity from \cite{mar}, according to which, two sets $A,B\subset\mathbb{C}$ are similar if they differ by a translation and a rotation. For each $T>0$, denote by $\hat{\gamma}_k(T) =:\gamma_k([T,\infty])$ the tail of the $k$-th spiroid. Then, the inverse of the Loewner flow, $g(\cdot,T):=f^{-1}(\cdot,T)$ maps $\mathbb{D}\setminus\gamma_k([0,T])$ onto $\mathbb{D}$.
	 \begin{figure}[ht]
	 	\centering
	 	\includegraphics[width=0.8\linewidth]{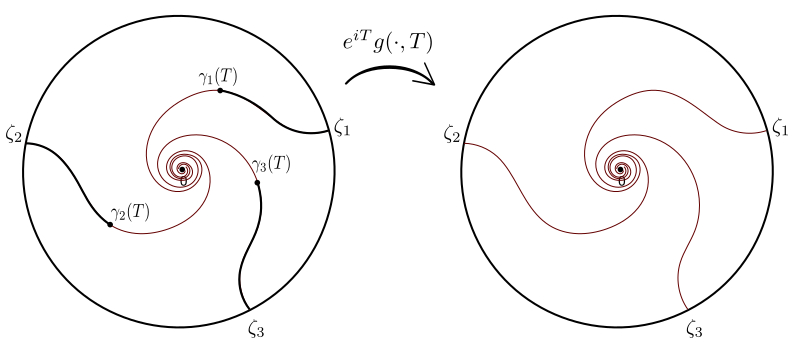}
	 	\caption{Self-similarity of the spiroids.}
	 	\label{fig:selfsimilar}
	 \end{figure}
  
It is clear that we can present a whole class of solutions $f(z,t)$, those driven by a function of the form of proposition $\ref{loewner-semigroup}$. We shall refer to such flows as \textit{spiroid} flows. An explicit flow, for instance, can be found in \cite{sola1}, where the driving measures $d\mu_t=\rho(\cdot,t)dt$, with densities $\rho(\theta,t)=2\sin^2(\pi\theta-ct)$, are being treated. The preceding family of measures corresponds to the driving function $p(z,t)=1-e^{2ict}/z$.

	 \section{Explicit solutions in the upper half plane.}
	With the machinery acquired from the previous section, we are able to transfer spiroid flows in the upper half plane and also, work similarly to deduce other cases of chordal Loewner flows. For this, given the real parameters $k_1,\dots ,k_n$ in increasing order and given $b_1,\dots,b_n>0$, we consider the chordal Loewner equation
	 \begin{equation}\label{chordalODE}
	 \dfrac{dw}{dt}(z,t)=\sum_{j=1}^{n}\dfrac{2b_j }{w(z,t)-k_j \sqrt{1-t}}
	 \end{equation}
	 for all $(z,t)\in\mathbb{H}\times[0,1)$, with initial value $w(z,0)=z$. For $n=1$, this case is studied in \cite{kad} and \cite{mar}. The technique to solve this equation will be to transform the right-hand part to a time-independent expression, as we do in the preceding chapter. Therefore, by taking the transform $v=(1-t)^{-1/2}w$, equation $(\ref{chordalODE})$ turns into the equation 
	 \begin{equation}\label{chordalODE2}
	 \dfrac{dv}{dt}=\dfrac{1}{2(1-t)}\left(v+\sum_{j=1}^{n}\dfrac{4b_j }{v-k_j}\right).
	 \end{equation}
	 which can be solved by separating variables.
	  Now, we will consider the polynomial
	 \begin{equation}\label{polynomial}
	 P(z):=z\prod_{j=1}^{n}(z-k_j)+\sum_{j=1}^{n}4b_j \prod_{i\neq j}(z-k_i)
	 \end{equation}
	 which is an $(n+1)$-degree polynomial with real coefficients. Thus, it has exactly $n+1$ complex roots. But because $P(k_m)=4b_m \prod_{j\neq m}(k_m-k_j)$, for $1\le m\le n$ and hence $\text{sgn}P(k_m )=(-1)^{n-m}$, there exist $\lambda_m\in(k_m ,k_{m+1})$ roots of $P$. So, since the $n-1$ roots of the polynomial are real and its coefficients are real, the remaining two roots can either be real, say $\rho_1, \rho_2 \in\mathbb{R}$ or non-real and conjugate, say $\beta, \bar{\beta}$ with $\beta\in\mathbb{H}$.
	 
	 \subsection{Spirals.} Assume the second case, thus, there exists some $\beta\in\mathbb{H}$, so that $P$ can be written in the form
	 \begin{equation}\label{polynomial2}
	 P(z)=\prod_{j=1}^{n-1}(z-\lambda_j)(z-\beta)(z-\bar{\beta}).
	 \end{equation}
	 Applying partial fraction decomposition, we introduce the numbers
	 \begin{equation}\label{chordalcoefficients2}
	     B=\dfrac{\prod_{j=1}^{n}(\beta-k_j)}{2i\text{Im}\beta\prod_{j=1}^{n-1}(\beta-\lambda_j)}, \quad A_j =\dfrac{\prod_{i=1}^{n}(\lambda_j-k_i)}{\prod_{i\neq j}(\lambda_j-\lambda_i)|\lambda_j-\beta|^2}<0,
	 \end{equation}
	 for $1\le j\le n-1$, for which the following relation holds:
	 \begin{equation}\label{chordalcoefficients}
	     \prod_{j=1}^{n}(z-k_j )=P(z)\left(\sum_{j=1}^{n-1}\dfrac{A_j }{z-\lambda_j}+\dfrac{B}{z-\beta}+\dfrac{\bar{B}}{z-\bar{\beta}}\right).
	 \end{equation}
	 Defining $B=:|B|e^{i\psi}$ and $\frac{A_j }{B}=-|\frac{A_j }{B}|e^{-i\psi}=:-a_j e^{-i\psi}$, then by $(\ref{chordalcoefficients})$, ODE $(\ref{chordalODE2})$ becomes
	 $$\left(\sum_{j=1}^{n-1}\dfrac{-a_j e^{-i\psi}}{v-\lambda_j}+\dfrac{1}{v-\beta}+\dfrac{e^{-2i\psi}}{v-\bar{\beta}}\right)dv=\dfrac{e^{-i\psi}}{2|B|(1-t)}dt.$$
	 Integrating the preceding formula, we deduce the implicit equation
	 \begin{equation}\label{chordalimplicit}
	     h(v(z,t))=(1-t)^{-\frac{1}{2|B|}e^{-i\psi}}h(z)
	 \end{equation}
	 for all $z\in\mathbb{H}$ and $t\in[0,1)$, where $h$ is given by
	 \begin{equation}\label{h(z)}
	     h(z)=\prod_{j=1}^{n-1}(z-\lambda_j)^{-a_j e^{-i\psi}}(z-\beta)(z-\bar{\beta})^{e^{-2i\psi}}.
	 \end{equation}
	 A straightforward differentiation and $(\ref{chordalcoefficients})$ gives us the derivative
	 \begin{equation}\label{h'(z)}
	     h'(z)=\dfrac{e^{-i\psi}}{|B|}\prod_{j=1}^{n-1}(z-\lambda_j)^{-a_j e^{-i\psi}-1}(z-\bar{\beta})^{e^{-2i\psi}-1} \prod_{j=1}^{n}(z-k_j ).
	 \end{equation}
	  In the proposition below we prove that $h$ maps the upper half plane onto a spirallike domain.
	  \begin{proposition}\label{chordalspirallike}
	      The function $h\in H(\mathbb{H})$ given by $(\ref{h(z)})$ is univalent and maps the upper half plane onto a $(-\psi)$-spirallike domain with respect to the origin, where the angle $\psi$ is determined by $(\ref{chordalcoefficients2})$ and it ranges in $(-\dfrac{\pi}{2},\dfrac{\pi}{2})$.
	  \end{proposition}
	  \begin{proof}
	  By $(\ref{h'(z)})$, we have that $h'$ is nonzero in $\mathbb{H}$ and $h$ is zero if and only if $z=\beta$. Then, applying proposition $\ref{spirallikeinH}$, due to the relations $(\ref{h(z)})$ and $(\ref{h'(z)})$, we have that
	  \begin{align*}
	     \text{Im}\left(e^{i\psi}\dfrac{(z-\beta)(z-\bar{\beta})h'(z)}{h(z)}\right)&=\dfrac{1}{|B|}\text{Im}\left(\dfrac{\prod_{j=1}^{n}(z-k_j )}{\prod_{j=1}^{n-1}(z-\lambda_j )}\right)\\
	      &=\dfrac{1}{|B|}\text{Im}\left(\dfrac{(z-\beta)(z-\bar{\beta})\prod_{j=1}^{n}(z-k_j )}{P(z)}\right)
      \end{align*}
	where the last equality is due to $(\ref{polynomial2})$. It suffices to show that for all $z\in\mathbb{H}$,
	  \begin{equation}\label{Hspirallike}
	  \text{Im}\left(\dfrac{\prod_{j=1}^{n-1}(z-\lambda_j )}{\prod_{j=1}^{n}(z-k_j )}\right)=\text{Im}\left(\dfrac{P(z)}{(z-\beta)(z-\bar{\beta})\prod_{j=1}^{n}(z-k_j )}\right)<0.
	  \end{equation}
	    
	     To see this, we first need to take into account that $\beta$ satisfies the relation $$\beta+\sum_{j=1}^{n}\frac{4b_j}{\beta-k_j}=0,$$
	since it is a root of the polynomial P. Then, the right-hand part of $(\ref{Hspirallike})$ is written as follows:
\begin{align*}
     &\dfrac{1}{2i\text{Im}\beta}\left(\dfrac{z}{z-\beta}-\dfrac{z}{z-\bar{\beta}}+\sum_{j=1}^{n}\dfrac{4b_j}{z-k_j}(\dfrac{1}{z-\beta}-\dfrac{1}{z-\bar{\beta}})\right)\\
     &\quad=\dfrac{1}{2i\text{Im}\beta}\left(\dfrac{\beta}{z-\beta}-\dfrac{\bar{\beta}}{z-\bar{\beta}}+\sum_{j=1}^{n}\dfrac{4b_j}{z-k_j}\dfrac{1}{z-\beta}-\sum_{j=1}^{n}\dfrac{4b_j}{z-k_j}\dfrac{1}{z-\bar{\beta}}\right)\\
	&\quad=\dfrac{1}{2i\text{Im}\beta}\left(\dfrac{\beta}{z-\beta}-\dfrac{\bar{\beta}}{z-\bar{\beta}}+\sum_{j=1}^{n}\dfrac{\frac{4b_j}{\beta-k_j}}{z-\beta}-\sum_{j=1}^{n}\dfrac{\frac{4b_j}{\beta-k_j}}{z-k_j}-\sum_{j=1}^{n}\dfrac{\frac{4b_j}{\bar{\beta}-k_j}}{z-\bar{\beta}}+\sum_{j=1}^{n}\dfrac{\frac{4b_j}{\bar{\beta}-k_j}}{z-k_j}\right)\\
	&\quad=\dfrac{1}{2i\text{Im}\beta}\left(\sum_{j=1}^{n}\dfrac{\frac{4b_j}{\bar{\beta}-k_j}}{z-k_j}-\sum_{j=1}^{n}\dfrac{\frac{4b_j}{\beta-k_j}}{z-k_j}\right)=\dfrac{1}{2i\text{Im}\beta}\sum_{j=1}^{n}4b_j\dfrac{\beta-\bar{\beta}}{(z-k_j)|\beta-k_j|^2}
\end{align*}
and each of the preceding summands has negative imaginary part for all $z\in\mathbb{H}$. 
 
 Therefore, $(\ref{Hspirallike})$ holds and we deduce that $h$ is $(-\psi)$-spirallike and the result follows. The fact that $-\pi/2<\psi<\pi/2$, follows directly from $(\ref{chordalcoefficients2})$ and $(\ref{Hspirallike})$.
	  \end{proof}
\begin{corollary}\label{chordalsolution}
Under the assumption that $P$ has a complex root $\beta\in\mathbb{H}$, the solution to Loewner's ODE $(\ref{chordalODE})$ is given by the formula
	  \begin{equation}\label{chordalexplicit}
	      w(z,t)=(1-t)^{1/2}h^{-1}((1-t)^{-\frac{1}{2|B|}e^{-i\psi}}h(z)),
	  \end{equation}
	 for all $z\in\mathbb{H}$ and $t\in[0,1)$, where $h$ is given by $(\ref{h(z)})$.
\end{corollary}
	    \begin{proof}
	         By $(\ref{chordalimplicit})$ and the univalence of $h$, $(\ref{chordalexplicit})$ is direct.
	    \end{proof}
	    
	     Our next step is to describe the geometry of the hulls produced by the chordal Loewner PDE, corresponding to $(\ref{chordalODE})$. By the corollary above, its unique solution is the Loewner flow in $\mathbb{H}\times[0,1)$,
	     	    $$f(z,t)=h^{-1}((1-t)^{\frac{1}{2|B|}e^{-i\psi}}h((1-t)^{-1/2}z)).$$

	      We shall show that $h$ maps the upper half plane onto the complement of $n$ logarithmic spirals, as in the previous section. Therefore, we study the behaviour of $h$ on the real line. A direct computation shows that	
	      \begin{equation}\label{boundaryh(z)}
	     h(x)=\exp[e^{-i\psi} S(x)-ie^{-i\psi}\sum_{j=1}^{n}a_j\text{arg}(x-\lambda_j)]
	      \end{equation}
	      for all $x$ in the extended real line, where $S$ is given by
	      \begin{equation}\label{S(x)}
	          S(x)=-\sum_{j=1}^{n}a_j\log|x-\lambda_j|+2\cos(\psi)\log|x-\beta|-2\sin(\psi)\text{arg}(x-\beta).
	      \end{equation}
	     This shows that the points $\lambda_1,\dots,\lambda_{n-1}$ are mapped to $\infty$. Note that $S(\pm\infty)=\infty$ as well, because $-\sum_{j=1}^{n-1}a_j+2\cos(\psi)=1/\abs B>0$, by comparing the coefficients in $(\ref{chordalcoefficients})$. This of course agrees with the hydrodynamic condition that $f(\cdot,t)$ satisfies. Thus, exactly $n$ points of $\overline{\mathbb{R}}$ are mapped to infinity. Differentiating and using relation $(\ref{chordalcoefficients})$ we also have that
	      \begin{align*}\label{S'(x)}
	          S'(x)&=-\sum_{j=1}^{n-1}\dfrac{a_j}{x-\lambda_j}+\dfrac{1+e^{-2i\psi}}{e^{-i\psi}}\dfrac{x-\text{Re}\beta}{|x-\beta|^2}+i\dfrac{1-e^{-2i\psi}}{e^{-i\psi}}\dfrac{\text{Im}\beta}{|x-\beta|^2}\\
	          &=\dfrac{1}{|B|}\dfrac{\prod_{j=1}^{n}(x-k_j)}{\prod_{j=1}^{n-1}(x-\lambda_j)|x-\beta|^2}
	      \end{align*}
	      which implies that $S'$ has constant sign, thus $S$ is monotonic, in each of the intervals $(-\infty,k_1)$, $(k_\mu ,\lambda_\mu)$, $(\lambda_\mu,k_{\mu+1})$ for $1\le\mu\le n-1$ and $(k_n,+\infty)$, and that the tip points of the spirals are the points $h(k_j)$, as expected by $(\ref{h'(z)})$.
	      
	      We are, now, ready to present the main result of this section.
	      \begin{theo}
	      Let $k_1,\dots,k_n$ be real points in increasing order and let $b_1,\dots,b_n$ be positive numbers, such that the polynomial $P$ given by $(\ref{polynomial})$, has a complex root $\beta\in\mathbb{H}$. Then, the solution to the chordal Loewner PDE in $\mathbb{H}\times[0,1)$,
	      \begin{equation}\label{chordalPDE}
	          \dfrac{\partial f}{\partial t}(z,t)=-f'(z,t)\sum_{j=1}^{n}\dfrac{2b_j}{z-k_j\sqrt{1-t}}
	      \end{equation}
	      with initial value $f(z,0)=z$, is the Loewner flow
	       \begin{equation}\label{chordalflow}
	    f(z,t)=h^{-1}((1-t)^{\frac{1}{2|B|}e^{-i\psi}}h((1-t)^{-1/2}z)),
	     \end{equation}
	      where $h$ is a $(-\psi)$-spirallike function of $\mathbb{H}$ given by $(\ref{h(z)})$.
	      
	    For each $j$, with $1\le j\le n$, the trace $\hat{\gamma}_j :=\{f(k_j\sqrt{1-t} ,t)/\ t\in[0,1)\}$ is a smooth curve lying in $\mathbb{H}$ that starts perpendicularly from $k_j$, spiralling about $\beta$. 
	      \end{theo}
	      \begin{proof}
	      By the discussion above, it only remains to show the second part. At first, define the logarithmic spirals $S_j: w=h(k_j)\exp(e^{-i\psi}t)$, $t\in\mathbb{R}$, with $1\le j\le n$. We already saw that $h$ maps the upper half plane onto the complement of the spirals $S^+_j: w=h(k_j)\exp(e^{-i\psi}t), t\ge0$, joining infinity with the tip points $h(k_j)$. As $t$ ranges in $[0,1)$, then looking at $(\ref{chordalflow})$ and following the radial case in section 3, the trajectory of the tip point $f(k_j\sqrt{1-t} ,t)$ is the inverse image of the spiral $S_j^-:=S_j\setminus S_j^+$ under $h$. The othogonality  statement follows from the work of S. Schleissinger  (see Theorem 1.2, \cite{slei1}).
	      
	      \begin{figure}[ht]
	              \centering
	              \includegraphics[width=0.7
	              \linewidth]{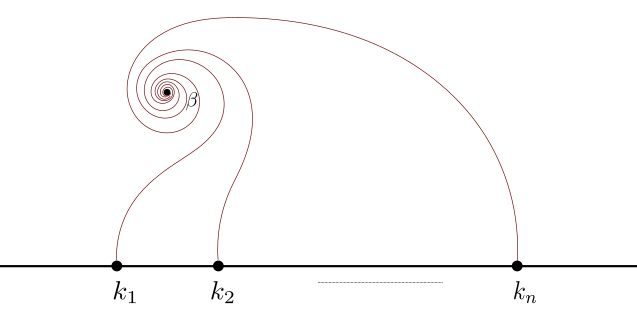}
	              \caption{Spiroids in the upper half plane, spiralling about the point $\beta$.}
	              \label{fig:SpiralsH5}
	          \end{figure}
	      
	      \end{proof}

\subsection{Non-tangential intersections.} We, now, proceed by assuming that the polynomial $(\ref{polynomial})$ has $n+1$ distinct roots in the real line. The ordering of the roots is important in the following analysis and for the sake of completeness we shall distinguish the cases below. We shall label the remaining roots of $P$ as $\rho_1$ and $\rho_2$.

\textbf{Case 1.} Assume without loss that $\rho_1<\rho_2$ and assume, also, that $\rho_1, \rho_2<k_1$, thus 
\begin{equation}\label{ordering1}
    \rho_1<\rho_2<k_1<\lambda_1<\dots<\lambda_{n-1}<k_n,
\end{equation}
thus, $P(z)=(z-\rho_1)(z-\rho_2)(z-\lambda_1)\dots(z-\lambda_{n-1})$. Therefore, by partial fraction decomposition we obtain that 
\begin{equation}\label{PFD2}
     \prod_{j=1}^{n}(z-k_j )=P(z)\left(\sum_{j=1}^{n-1}\dfrac{A_j }{z-\lambda_j}+\dfrac{B_1}{z-\rho_1}+\dfrac{B_2}{z-\rho_2}\right)
\end{equation}
where the parameters above are given as 
\begin{equation}\label{chordalparameters1}
B_1=\dfrac{\prod_{j=1}^{n}(\rho_1-k_j)}{(\rho_1-\rho_2)\prod_{j=1}^{n-1}(\rho_1-\lambda_j)}>0, \quad B_2=\dfrac{\prod_{j=1}^{n}(\rho_2-k_j)}{(\rho_2-\rho_1)\prod_{j=1}^{n-1}(\rho_2-\lambda_j)}<0    
\end{equation}
and
\begin{equation}\label{chordalparameters2}
	     A_j =\dfrac{\prod_{i=1}^{n}(\lambda_j-k_i)}{\prod_{i\neq j}(\lambda_j-\lambda_i)(\lambda_j-\rho_1)(\lambda_j-\rho_2)}<0, \quad B_1+B_2+\sum_{j=1}^{n-1}A_j =1.
	 \end{equation}
	 for $1\le j\le n-1$.
	 If we also set $b:=-B_2/B_1>0$, $a_j:=-A_j/B_1>0$, then, ODE $(\ref{chordalODE2})$ is written as 
	 $$\left(\sum_{j=1}^{n-1}\dfrac{a_j}{v-\lambda_j}-\dfrac{1}{v-\rho_1}+\dfrac{b}{v-\rho_2}\right)dv=-\dfrac{dt}{2B_1(1-t)},$$
	 in which case the solution to $(\ref{chordalODE})$ is implicitly given by the equation
	 \begin{equation}\label{chordalimplicit2}
	     h((1-t)^{-\frac{1}{2}}w(z,t))=(1-t)^{\frac{1}{2B_1}}h(z),
	 \end{equation}
	 where $h$ is given by
	 \begin{equation}\label{h(z)2}
	     h(z)=\dfrac{(z-\rho_2)^b\prod_{j=1}^{n-1}(z-\lambda_j)^{a_j}}{z-\rho_1}.
	 \end{equation}
	 
Now, in order to write the solution explicitly we have to prove that $h$ is invertible. To see this, consider the Möbius transform $T(z)=\frac{z-\rho_2}{z-\rho_1}:\mathbb{H}\rightarrow\mathbb{H}$ and note also that $-1+b+\sum a_j=-1/B_1$, by $(\ref{chordalparameters2})$. It is then straightforward to calculate 
\begin{equation}\label{h/C}
    h\circ T^{-1}(z)=C(z-1)^{\frac{1}{B_1}}z^b\prod_{j=1}^{n-1}(z-T(\lambda_j))^{a_j},
\end{equation}
but this is a Schwarz-Cristoffel transform of the upper half plane and as a result we deduce that $h$ is univalent and maps $\mathbb{H}$ onto a rotation of $\mathbb{H}$ minus $n$ straight line segments emanating from the origin. Notice that the tip points of these slits are the point $h(k_j)$, because we have by $(\ref{PFD2})$ that
$$h'(z)=-h(z)\dfrac{\prod_{j=1}^{n}(z-k_j)}{B_1 P(z)},$$
thus the derivative is zero at exactly the points $k_1,\dots,k_n$.

\textbf{Case 2.} Consider the case where $k_n<\rho_1, \rho_2$ and by relabeling assume that $\rho_2<\rho_1$. Relations $(\ref{PFD2})-(\ref{chordalimplicit2})$ still hold and $h$ is given as above.

\textbf{Case 3.} Assume, now, that at least one of the two roots $\rho_1,\rho_2$ lies in some interval $(k_{\mu},k_{\mu+1})$. Recall that $\text{sgn}P(k_j)=(-1)^{n-j}$, hence both roots have to lie in this interval. In addition, for a reason that will become clear in remark $\ref{chordalremark}$, we relabel so that $\lambda_{\mu}<\rho_1<\rho_2$. Under the preceding assumption, this case is treated the same way as the previous ones, relations $(\ref{PFD2})-(\ref{chordalimplicit2})$ still hold and $h$ is given as in case 1.

In each case, the Loewner flow is given by the formula
\begin{equation}\label{f(z,t)2}
    f(z,t)=h^{-1}((1-t)^{-\frac{1}{2B_1}}h((1-t)^{-\frac{1}{2}}z))\end{equation}
for all $z\in\mathbb{H}$ and $t\in[0,1)$. We have that $h$ maps $\mathbb{H}$ onto $\mathbb{H}\setminus\bigcup_{j=1}^{n}S_j$, where $S_j$ are straight line segments emanating from the origin with tip points $h(k_j)$. Therefore, in order to keep track of the orbits of the tip points $f(k_j\sqrt{1-t},t)$, we have that as $t$ tends to $1$, then $(1-t)^{-1/2B_1}h(k_j)$ tend to infinity, but because $h(\rho_1)=\infty$, this yields that the tip points are attracted to $\rho_1\in\mathbb{R}$.  
\begin{rmk}\label{chordalremark}
 It is clear by $(\ref{f(z,t)2})$, that for all $z\in\mathbb{H}$, $f(z\sqrt{1-t},t)\rightarrow\rho_1$ as $t\rightarrow1$. This means that the point of attraction for the tip points of the flow is one of the roots of $P$. So the question is which of the $n+1$ roots of $P$ will play the role of the attraction point. Of course, the choice of this point is not arbitrary, rather it is determined by $(\ref{chordalparameters1})$. We see that all the introduced parameters are negative except one. The point corresponding to the positive parameter will be the point of attraction.
\end{rmk}
We therefore proved the following result.
\begin{theo}
Let $k_1,\dots,k_n$ be real points in increasing order and let $b_1,\dots,b_n$ be positive numbers, such that the polynomial $P$ given by $(\ref{polynomial})$, has distinct real roots. Then, the solution to the chordal Loewner PDE in $\mathbb{H}\times[0,1)$,
	      $$\dfrac{\partial f}{\partial t}(z,t)=-f'(z,t)\sum_{j=1}^{n}\dfrac{2b_j}{z-k_j\sqrt{1-t}}$$
	      with initial value $f(z,0)=z$, is the Loewner flow
\begin{equation}\label{chordalflow2}
f(z,t)=h^{-1}((1-t)^{-\frac{1}{2B_1}}h((1-t)^{-1/2}z)),
	     \end{equation}
	      where $h$ is a Schwarz-Cristoffel transform of $\mathbb{H}$, given by $(\ref{h(z)2})$.
	      
	    For each $j$, with $1\le j\le n$, the trace $\hat{\gamma}_j :=\{f(k_j\sqrt{1-t} ,t)/\ t\in[0,1)\}$ is a smooth curve lying in $\mathbb{H}$ that starts perpendicularly from $k_j$, intersecting the real line at the root $\rho_1$ (determined by remark $\ref{chordalremark}$), with non-zero angle. 
	      \end{theo}
\begin{proof}
 It remains to prove that the traces intersect the real line non-tangentially. For this, consider without loss of generality that $\lambda_{\mu}<\rho_1<\rho_2<k_{\mu+1}$, for some $\mu$, as in the third case. By $(\ref{h/C})$, consider the function $\frac{h(z)}{C}$, which maps $\mathbb{H}$ onto $\mathbb{H}\setminus\bigcup_{j=1}^{n}[0,h(k_j)/C]$ as in the figure below (it would be enough to assume that C=1).
 
 \begin{figure}[ht]
	              \centering
	              \includegraphics[width=1
	              \linewidth]{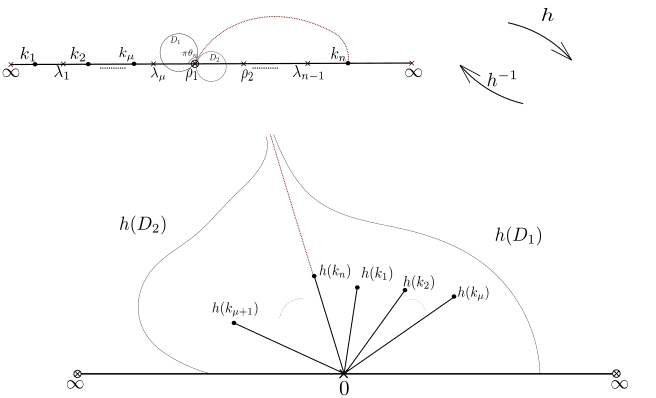}
	              \caption{The image of $h$ for $\lambda_{\mu}<\rho_1<\rho_2<k_{\mu+1}$.}
	              \label{fig:tangentialintersection1}
	          \end{figure}
 
  Let $\theta_j\pi$, for $1\le j\le n$, be the angles formed by the line segments $[0,h(k_j)]$ with the real line. We will prove that each trace $\hat{\gamma}_j$ intersects $\mathbb{R}$ under angle $(1-\theta_j)\pi$, using the following geometric argument. Consider discs $D_1, D_2$ of radius $\epsilon$, centered at $\rho_1+i\epsilon e^{i(1-\theta_j)\pi}$ and $\rho_1-i\epsilon e^{i(1-\theta_j)\pi}$ respectively, and we look at their images under $\frac{h}{C}$. For $\zeta_1^j(\theta)=\rho_1+i\epsilon e^{i(1-\theta_j)\pi}+\epsilon e^{i\theta}\in\partial D_1\cap\mathbb{H}$, we have
  \begin{align*}
      F_1^j(\theta):=&\text{arg}\left(\dfrac{h(\zeta_1^j(\theta))}{C}\right)=(1-b-\sum_{j=\mu+1}^{n-1}a_j)\pi\\
      +&b\text{Arg}(\rho_1-\rho_2+\epsilon(\cos(\theta)-\sin(\theta_j\pi))+i\epsilon(\sin(\theta) -\cos(\theta_j\pi)))\\
      +&\sum_{j=1}^{n-1}a_j\text{Arg}(\rho_1-\lambda_j+\epsilon(\cos(\theta)-\sin(\theta_j\pi))+i\epsilon(\sin(\theta) -\cos(\theta_j\pi)))\\
      -&\text{Arg}(\cos(\theta)-\sin(\theta_j\pi)+i(\sin(\theta) -\cos(\theta_j\pi)))\\
      =&(1-b-\sum_{j=\mu+
      1}^{n-1}a_j)\pi+b\text{Arccot}\left(\dfrac{\rho_1-\rho_2+\epsilon(\cos(\theta)-\sin(\theta_j\pi))}{\epsilon(\sin(\theta) -\cos(\theta_j\pi))}\right)\\
      +&\sum_{j=1}^{n-1}a_j\text{Arccot}\left(\dfrac{\rho_1-\rho_2+\epsilon(\cos(\theta)-\sin(\theta_j\pi))}{\epsilon(\sin(\theta) -\cos(\theta_j\pi))}\right)-\dfrac{\pi}{2}-\dfrac{\theta-\frac{\pi}{2}+(1-\theta_j)\pi}{2}.
      \end{align*}
      Notice that the left endpoint of $\partial D_1\cap\mathbb{H}$ is mapped onto the real line, while for the right endpoint we have that $F_1^j(-\frac{\pi}{2}+(1-\theta_j)\pi)=\theta_j\pi$, the angle of the $j$-th line segment. Differentiating with respect to $\theta$, we also have that
      \begin{align*}
      (F_1^j)'(\theta)=&-\dfrac{1}{2}+b\epsilon\dfrac{\epsilon(1+\sin((\theta_j-\theta)\pi))+\cos(\theta)(\rho_1-\rho_2)}{\epsilon^2(\sin(\theta)+\cos(\theta_j\pi))^2+(\rho_1-\rho_2+\epsilon(\cos(\theta)-\sin(\theta_j\pi)))^2}\\
      +&\sum_{j=1}^{n-1}a_j\epsilon\dfrac{\epsilon(1+\sin((\theta_j-\theta)\pi))+\cos(\theta)(\rho_1-\lambda_j)}{\epsilon^2(\sin(\theta)+\cos(\theta_j\pi))^2+(\rho_1-\lambda_j+\epsilon(\cos(\theta)-\sin(\theta_j\pi)))^2},
      \end{align*}
      which is negative for sufficiently small $\epsilon$. Therefore, $h(\partial D_1\cap\mathbb{H})$ is a curve starting from the positive half-line, with decreasing argument and extending to infinity asymptotically with respect to $\{xe^{i\theta_j}/ x\ge0\}$, as seen in figure $\ref{fig:tangentialintersection1}$. We do the same for the disc $D_2$. Let $\zeta_2^j(\theta):=\rho_1-i\epsilon e^{i(1-\theta_j)\pi}+\epsilon e^{i\theta}\in\partial D_2\cap\mathbb{H}$ and $F_2^j(\theta):=\text{arg}(\frac{h(\zeta_2^j(\theta))}{C})$. We then have that $F_2^j(\dfrac{\pi}{2}+(1-\theta_j)\pi)=\theta_j\pi$ and as before, its derivative is negative in the interval $[(F_2^j)^{-1}(\pi),\frac{\pi}{2}+(1-\theta_j)\pi]$. Thus, the image of $\partial D_2\cap\mathbb{H}$ is a curve emanating from the negative half-line with increasing argument, with the line $\{xe^{i\theta_j}/ x\ge0\}$ being its asymptote as it extends to infinity.
      
      Notice, finally, that $D_j\cap\mathbb{H}$ are mapped to the subdomains determined by the two curves above, not containing the line segments $[0,h(k_i)]$. This means that the orbit of $k_j$, approaches $\rho_1$ without intersecting the discs $D_j$, hence it intersects the real line with angle $(1-\theta_j)\pi$, which completes the proof.
      \begin{figure}[ht]
	              \centering
	              \includegraphics[width=0.9
	              \linewidth]{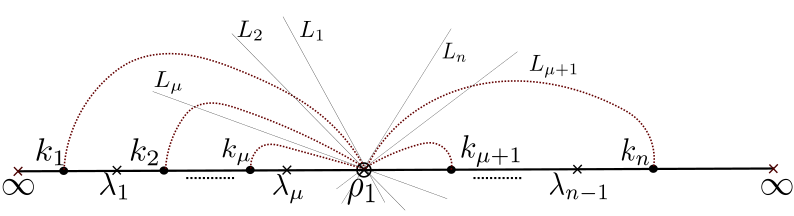}
	              \caption{The orbits of the driving functions $k_j\sqrt{1-t}$.}
	              \label{fig:tangentialintersection1.5}
	          \end{figure}
      \end{proof}

      \subsection{Tangential intersections.} In the final case of our study, we consider the case where $P$ has a multiple root, say $\rho_0\in\mathbb{R}$. It is possible to have that $\rho_0$ coincides with some $\lambda_{\mu}$ and  this implies that it is actually a triple root. To begin with, we consider the case where $\rho_0$ is a root of order $2$. As in the previous section, we distinguish the cases where $\rho_0<k_1$, or $\rho_0\in(k_{\mu},\lambda_{\mu})$, or $\rho_0\in(\lambda_{\mu},k_{\mu+1})$, or finally $\rho_0>k_n$. Each case is treated the same way.

\textbf{Case 1.} Assume that $P(z)=(z-\rho_0)^2\prod_{j=1}^{n-1}(z-\lambda_j)$, with $\rho_0<k_1$. Now, by partial fraction decomposition we have that 
\begin{equation}
    \label{PFD3}
     \prod_{j=1}^{n}(z-k_j )=P(z)\left(\sum_{j=1}^{n-1}\dfrac{A_j }{z-\lambda_j}+\dfrac{B_1}{z-\rho_0}+\dfrac{B_2z}{(z-\rho_0)^2}\right)
\end{equation}
from which we deduce that the parameters are given by 
\begin{equation}\label{PFDparameters1}
     A_j =\dfrac{\prod_{i=1}^{n}(\lambda_j-k_i)}{\prod_{i\neq j}(\lambda_j-\lambda_i)(\lambda_j-\rho_0)^2}<0, \quad  B_2\rho_0=\dfrac{\prod_{j=1}^{n}(\rho_0-k_j)}{\prod_{j=1}^{n-1}(\rho_0-\lambda_j)}<0
\end{equation}
for $1\le j\le n-1$, and furthermore $B_1+B_2+\sum_{j=1}^{n-1}A_j =1$.
Hence, ODE $(\ref{chordalODE2})$ is written as 
$$\left(\sum_{j=1}^{n-1}\dfrac{A_j}{v-\lambda_j}+\dfrac{B_1+B_2}{v-\rho_0}+\dfrac{B_2\rho_0}{(v-\rho_0)^2}\right)dv=\dfrac{dt}{2(1-t)}$$
and then the solution is implicitly defined by the equation
\begin{equation}\label{chordalimplicit3}
	     h((1-t)^{-\frac{1}{2}}w(z,t))=\dfrac{1}{2}\log(1-t)+h(z),
	 \end{equation}
 where $h$ is given by
	 \begin{equation}\label{h(z)3}
	     h(z)=\sum_{j=1}^{n-1}A_j\log(z-\lambda_j)+(1-\sum_{j=1}^{n-1}A_j)\log(z-\rho_0)-\dfrac{B_2\rho_0}{z-\rho_0}.
	 \end{equation}
	 
	 \textbf{Case 2.} For the other cases, we have that relations $(\ref{PFD3})-(\ref{h(z)3})$ still hold if we assume that $\rho_0\in(k_{\mu},k_{\mu+1})$ or $k_n<\rho_0$, with only difference in $(\ref{PFDparameters1})$, that $B_2\rho_0>0$ when $\rho_0\in(k_{\mu},\lambda_{\mu})$ or $k_n<\rho_0$. 
	 \begin{proposition}
	     The function $h\in H(\mathbb{H})$ given by $(\ref{h(z)3})$ is univalent and maps the upper half plane onto a half plane determined by a translation of $\mathbb{R}$, minus $n$ half-lines parallel to $\mathbb{R}$.
	     \end{proposition}
	     \begin{proof}
	         Assume that $\rho_0<k_1$. Applying the Möbius transform $T(z)=\frac{1}{\rho_0-z}:\mathbb{H}\rightarrow\mathbb{H}$ we have that
	          $$h\circ T^{-1}(z)=\sum_{j=1}^{n-1}A_j\log\left((\rho_0-\lambda_j)\dfrac{z-T(\lambda_j)}{z}\right)+(1-\sum_{j=1}^{n-1}A_j)\log\left(\dfrac{-1}{z}\right)+B_2\rho_0 z$$
	          and a differentiation shows that
	          $$(h\circ T^{-1})'(z)=\sum_{j=1}^{n-1}A_j\dfrac{1}{z-T(\lambda_j)}-\dfrac{1}{z}+B_2\rho_0 .$$
	          This implies that $\text{Im}((h\circ T^{-1})'(z))>0$, for all $z\in\mathbb{H}$, but because $\mathbb{H}$ is a convex domain, we deduce that $h$ is univalent. We argue similarly for the other cases, by applying the same Möbius transform if $k_n<\rho_0$, or $T_2(z)=\frac{z-\lambda_{\mu}}{z-\rho_0}$ if $k_{\mu}<\rho_0<\lambda_{\mu}$, or $-T_2(z)$ if $\lambda_{\mu}<\rho_0<k_{\mu+1}$. 
	          
	          For the second part we only have to find the image of the real line under $h$. We have that
	          \begin{eqnarray*}
	              h(x)&=&\sum_{j=1}^{n-1}A_j\log|x-\lambda_j|+(1-\sum_{j=1}^{n-1}A_j)\log|x-\rho_0|-\dfrac{B_2\rho_0}{x-\rho_0}\\
	              &+&i[\sum_{j=1}^{n-1}A_j\text{Arg}(x-\lambda_j)+(1-\sum_{j=1}^{n-1}A_j)\text{Arg}(x-\rho_0)]
	          \end{eqnarray*}
	          for all $x\in\mathbb{R}$, and the result follows.
	          
	           \begin{figure}[ht]
	              \centering
	              \includegraphics[width=1
	              \linewidth]{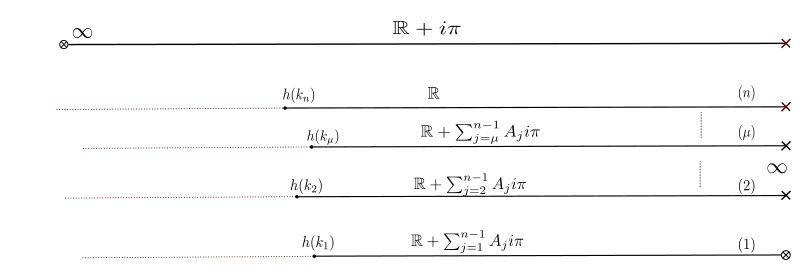}
	              \caption{The image of $h$ for $\rho_0<k_1$.}
	              \label{fig:half lines 2}
	          \end{figure}
	          \begin{figure}[ht]
	              \centering
	              \includegraphics[width=1
	              \linewidth]{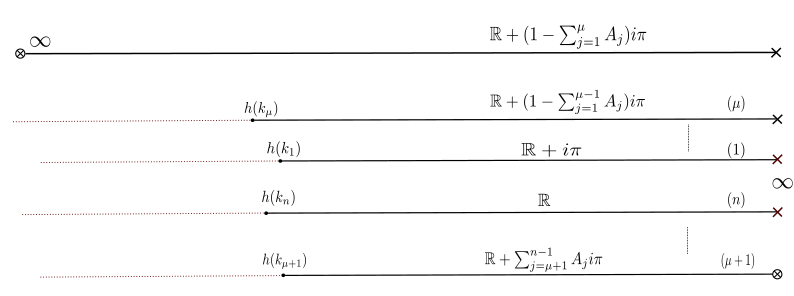}
	              \caption{The image of $h$ for $\lambda_{\mu}<\rho_0<k_{\mu+1}$.}
	              \label{fig:half lines 3}
	          \end{figure}
	     \end{proof}
	     
	     \begin{corollary}\label{chordalsolution3}
The solution to Loewner's ODE $(\ref{chordalODE})$, under the assumptions of this section, is given by the formula
	  \begin{equation}\label{chordalexplicit3}
	      w(z,t)=(1-t)^{1/2}h^{-1}(-\dfrac{1}{2}\log(1-t)+h(z)),
	  \end{equation}
	 for all $z\in\mathbb{H}$ and $t\in[0,1)$, where $h$ is given by $(\ref{h(z)3})$.
\end{corollary}
	   It is clear by figures $\ref{fig:half lines 2}$ and $\ref{fig:half lines 3}$ that the point at infinity has $n+1$ preimages in $\overline{\mathbb{R}}$, or in the language of Caratheodory's theory, it has $n+1$ prime ends. Therefore, if we consider the orbits $x+h(k_j)$ as $x\rightarrow-\infty$, then their preimages can accumulate to any of the points $\rho_0,\lambda_1,\dots,\lambda_{n-1}$ or $\infty$.
	   
	   Consider the case where $k_{\mu}<\rho_0<\lambda_{\mu}$. 
	   To see that $\rho_0$ is the accumulation point, take $C:=\overline{D(\rho_0,\epsilon)\cap\mathbb{H}}$ to be the closure of a half-disc with center $\rho_0$, so that it does not contain any point $k_j,\lambda_j$. Then, $h(\partial D(\rho_0,\epsilon))$ is a Jordan arc with endpoints $h(\rho_0-\epsilon)\in\mathbb{R}+(1-\sum_{j=1}^{\mu-1}A_j)i\pi$ and $h(\rho_0+\epsilon)\in\mathbb{R}+\sum_{j=\mu}^{n-1}A_ji\pi$, as in the figure below.\begin{figure}[ht]
	              \centering
	              \includegraphics[width=1
	              \linewidth]{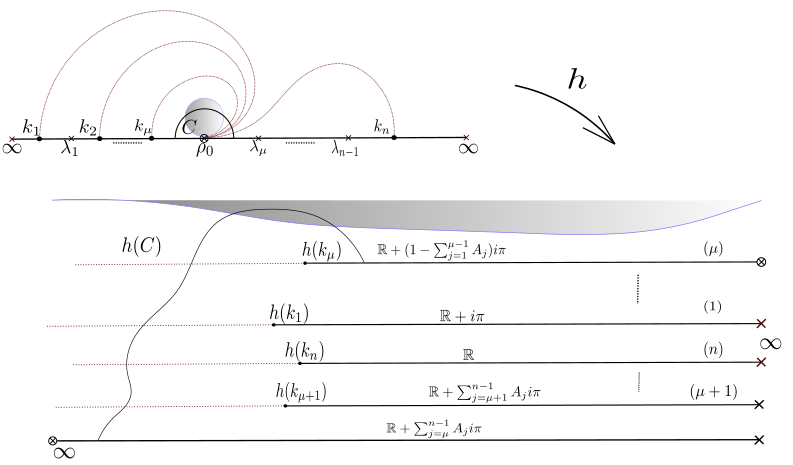}
	              \caption{The image of $h$ for $k_{\mu}<\rho_0<\lambda_{\mu}$.}
	              \label{fig:half lines 1 tangential}
	          \end{figure}	 
	          This implies that $h(C)$ is one of the two domains determined by this curve, but since $C$ does not contain any of the $\lambda_j$'s, $h(C)$ is the domain not containing the $n$ half-lines.

We secondly prove that the traces intersect the real line at $\rho_0$ tangentially, using a simple geometric argument. Consider the disc $D:=D(\rho_0+\delta i,\delta)$, tangent to $\mathbb{R}$ at $\rho_0$. Then, a direct computation gives us that
$$\text{Im}(h(\zeta))=\sum_{j=1}^{n-1}A_j\text{Arg}(\zeta-\lambda_j)+(1-\sum_{j=1}^{n-1}A_j)\text{Arg}(\zeta-\rho_0)+\dfrac{B_2\rho_0}{2\delta},$$
for all $\zeta\in\partial D(\rho_0+\delta i,\delta)$. Hence, for $\delta$ small enough, $\text{Im}(h(\zeta))>(1-\sum_{j=1}^{\mu-1}A_j)\pi$, which means that the image of the disc lies above the line $\mathbb{R}+(1-\sum_{j=1}^{\mu-1}A_j)i\pi$ as in figure $\ref{fig:half lines 1 tangential}$. Therefore, the traces approach $\rho_0$ without intersecting the disc, thus tangentially.

To conclude, we observe that as $t$ tends to $1$, the traces approach $\rho_0$ from only one direction, thus, $\text{Arg}(\gamma_j(t)-\rho_0)\rightarrow0$ for all $1\le j\le n$, or $\text{Arg}(\gamma_j(t)-\rho_0)\rightarrow\pi$ for all $1\le j\le n$. Indeed, for any $\epsilon$ small enough, the image of the part of $\partial C$ from its right endpoint to the point of $D$ it intersects first, crosses all orbits $\{x+h(k_j)/x\ge0\}$, as in the figure above. Note that the other part of $C$ cannot intersect some trace for those $\epsilon$'s, as this would violate the No Koebe Arcs Theorem (see \cite{pom}).

\subsection{Orthogonal intersections.}
	   Consider, now, the case where $\rho_0=\lambda_{\mu}$, for some $\mu$, hence it is a triple root of $P$. Partial fraction decomposition is written as 
	   \begin{equation}
    \label{PFD4}
     \prod_{j=1}^{n}(z-k_j )=P(z)\left(\sum_{j\neq\mu}\dfrac{A_j }{z-\lambda_j}+\dfrac{B_1}{z-\rho_0}+\dfrac{B_2z}{(z-\rho_0)^2}+\dfrac{B_3z^2}{(z-\rho_0)^3}\right)
\end{equation}
and hence we get that
\begin{equation}\label{PFDparameters2}
     A_j =\dfrac{\prod_{i=1}^{n}(\lambda_j-k_i)}{\prod_{i\neq j,\mu}(\lambda_j-\lambda_i)(\lambda_j-\rho_0)^3}<0, \quad  B_3\rho_0^2=\dfrac{\prod_{j=1}^{n}(\rho_0-k_j)}{\prod_{j\neq\mu}(\rho_0-\lambda_j)}<0
\end{equation}
	   and that $B_1+B_2+B_3+\sum_{j\neq\mu}A_j=1$. Now, ODE $(\ref{chordalODE2})$ is equivalent to
$$\left(\sum_{j\neq\mu}\dfrac{A_j}{v-\lambda_j}+\dfrac{B_1+B_2+B_3}{v-\rho_0}+\dfrac{(B_2+2B_3)\rho_0}{(v-\rho_0)^2}+\dfrac{B_3\rho_0^2}{(v-\rho_0)^3}\right)dv=\dfrac{dt}{2(1-t)}$$
and then the solution is implicitly given by the equation
\begin{equation}\label{chordalimplicit4}
	     h((1-t)^{-\frac{1}{2}}w(z,t))=-\dfrac{1}{2}\log(1-t)+h(z),
	 \end{equation}
 where $h$ is given by (set $C=(B_2+2B_3)\rho_0$ as it will not play any important role)
	 \begin{equation}\label{h(z)4}
	     h(z)=\sum_{j\neq\mu}A_j\log(z-\lambda_j)+(1-\sum_{j\neq\mu}A_j)\log(z-\rho_0)-\dfrac{C}{z-\rho_0}-\dfrac{B_3\rho_0^2}{(z-\rho_0)^2}.
	 \end{equation}

 \begin{proposition}
	     The function $h\in H(\mathbb{H})$ given by $(\ref{h(z)4})$ is univalent and maps the upper half plane onto the complement of $n$ half-lines parallel to $\mathbb{R}$.
	     \end{proposition}
\begin{proof}
 Applying the Möbius transform $T(z)=\frac{1}{\rho_0-z}:\mathbb{H}\rightarrow\mathbb{H}$ we have that
	          \begin{align*}
	              h\circ T^{-1}(z)&=\sum_{j\neq\mu}A_j\log\left((\rho_0-\lambda_j)\dfrac{z-T(\lambda_j)}{z}\right)+(1-\sum_{j\neq\mu}A_j)\log\left(\dfrac{-1}{z}\right) \\
	              &+Cz-\dfrac{1}{2}B_3\rho_0^2z^2
	          \end{align*}
	          and by differentiating we have that
	          $$(h\circ T^{-1})'(z)=\sum_{j\neq\mu}A_j\dfrac{1}{z-T(\lambda_j)}-\dfrac{1}{z}+C-B_3\rho_0^2z.$$
	          Then, by $(\ref{PFDparameters2})$ we have that $\text{Im}((h\circ T^{-1})'(z))>0$ for all $z\in\mathbb{H}$ and this implies that $h$ is univalent.

	          For the second part, we only have to find $h(\mathbb{R})$, shown in figure $\ref{fig:half lines 4}$, and the result follows.
	          
	             \begin{figure}[ht]
	              \centering
	              \includegraphics[width=1
	              \linewidth]{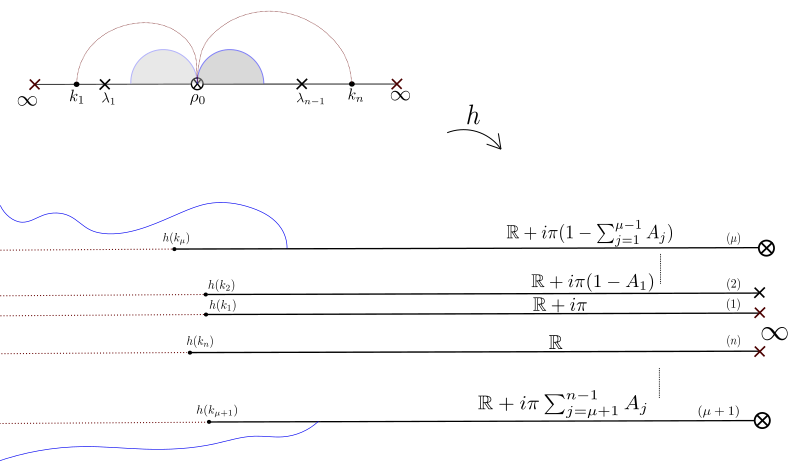}
	              \caption{The image of $h$ for a triple root.}
	              \label{fig:half lines 4}
	          \end{figure}	      
	         
\end{proof}
Arguing as in the previous section, we can see that the traces of the flow in this case accumulate at $\rho_0$, while considering the half-discs $D(\rho_0\pm\epsilon,\epsilon)$ as in figure $\ref{fig:half lines 4}$, we also deduce that they do not intersect these discs. Indeed, consider $\epsilon>0$, so small that $\epsilon C+B_3\rho_0<0<\epsilon C-B_3\rho_0$ and that the preceding discs do not contain any of the $k_j$'s. Then, for $\zeta(\theta)=\rho_0-\epsilon+\epsilon e^{i\theta}\in\partial D(\rho_0-\epsilon,\epsilon)$, we have that
\begin{align*}
\text{Re}(h(\zeta(\theta)))&=\sum_{j\neq\mu}A_j\log|\zeta(\theta)-\lambda_j|+(1-\sum_{j\neq\mu}A_j)\log|\zeta(\theta)-\rho_0|+\dfrac{C}{2\epsilon}\\
&+\dfrac{B_3\rho_0^2}{4\epsilon^2}\dfrac{\cos(\theta)}{\sin^2(\frac{\theta}{2})},
\end{align*}
which tends to $-\infty$ as $\theta\rightarrow0^+$ and 
\begin{align*}
    \text{Im}(h(\zeta(\theta)))&=\sum_{j\neq\mu}A_j\text{Arg}(\zeta(\theta)-\lambda_j)+(1-\sum_{j\neq\mu}A_j)\text{Arg}(\zeta(\theta)-\rho_0)\\
    &+\dfrac{C\epsilon-B_3\rho_0^2}{2\epsilon^2}\cot(\frac{\theta}{2}),
\end{align*}
which tends to $+\infty$ as $\theta\rightarrow0^+$. This implies that for all $\zeta\in\partial D(\rho_0-\epsilon,\epsilon)$, sufficiently close to $\rho_0$, $\text{Im}(h(\zeta))>\pi(1-\sum_{j=1}^{\mu-1}A_j)$. And because the image $h(D(\rho_0-\epsilon,\epsilon))$ does not intersect the half-lines near the point at infinity, we have that the traces $f(k_j\sqrt{1-t},t)$ do not intersect the disc $D(\rho_0-\epsilon,\epsilon)$ for $t$ near $1$. Similarly for the disc $D(\rho_0+\epsilon,\epsilon)$.
Therefore, they approach $\rho_0$ orthogonally.
	     \begin{corollary}\label{chordalsolution4}
The solution to Loewner's ODE $(\ref{chordalODE})$, under the assumptions of this section, is given by the formula
	  \begin{equation}\label{chordalexplicit4}
	      w(z,t)=(1-t)^{1/2}h^{-1}(-\dfrac{1}{2}\log(1-t)+h(z)),
	  \end{equation}
	 for all $z\in\mathbb{H}$ and $t\in[0,1)$, where $h$ is given by $(\ref{h(z)4})$.
\end{corollary}
To conclude, we therefore proved the following result.
\begin{theo}
Let $k_1,\dots,k_n$ be real points in increasing order and let $b_1,\dots,b_n$ be positive numbers, such that the polynomial $P$ given by $(\ref{polynomial})$, has a multiple root, either double or triple. Then, the solution to the chordal Loewner PDE in $\mathbb{H}\times[0,1)$,
	      
	      $$\dfrac{\partial f}{\partial t}(z,t)=-f'(z,t)\sum_{j=1}^{n}\dfrac{2b_j}{z-k_j\sqrt{1-t}}$$
	      with initial value $f(z,0)=z$, is the Loewner flow
	       \begin{equation}\label{chordalflow3}
	    f(z,t)=h^{-1}(\dfrac{1}{2}\log(1-t)+h((1-t)^{-1/2}z)),
	     \end{equation}
	      where $h$ is given by $(\ref{h(z)3})$ if the root is of order $2$ or by $(\ref{h(z)4})$ if the root is of order $3$.
	      
	    For each $j$, with $1\le j\le n$, the trace $\hat{\gamma}_j :=\{f(k_j\sqrt{1-t} ,t)/\ t\in[0,1)\}$ is a smooth curve lying in $\mathbb{H}$ that starts perpendicularly from $k_j$, intersecting the real line at the root $\rho_0$ tangentially if it is of order 2 or orthogonally if it is of order 3. 
	      \end{theo}

\subsection{A remark on semigroups.} As a conclusion, we will unfold the semigroup nature of the Loewner flows produced by the driving functions of this chapter. Consider the reparameterization $e^{-\tau}=\sqrt{1-t}$, for $\tau\ge0$. Then, the flow $(\ref{chordalflow})$ is written as
\begin{equation}\label{chordalsemigroup0}
    f(z,t)=h^{-1}(e^{-\frac{t}{|B|}e^{-i\psi}}h(e^t z)),
\end{equation}
for all $z\in\mathbb{H}$ and $t\ge0$. It is, therefore, obvious that $\tilde{f}(z,t)=f(e^{-t}z,t)$ is an elliptic semigroup of holomorphic self maps of $\mathbb{H}$, with Denjoy-Wolff point $\beta$. In view of the latter, the chordal equation $(\ref{chordalPDE})$ becomes
\begin{equation}\label{chordalsemigroup}
    \dfrac{\partial\tilde{f}}{\partial t}(z,t)=-\tilde{f}'(z,t)(z+\sum_{j=1}^{n}\dfrac{4b_j}{z-k_j})=:-\tilde{f}'(z,t)P_{\mathbb{H}}(z)
\end{equation}
for all $z\in\mathbb{H}$ and $t\ge0$. We set $P_{\mathbb{H}}$ to play the role of the infinitesimal generator for the semigroup. Notice, also, that by $(\ref{chordalcoefficients2})$ we write $B=1/P_{\mathbb{H}}'(\beta)$.

Assume, now, the Möbius transform $Tz=\frac{z-\beta}{z-\overline{\beta}}$. It is natural to consider the conjugation $g(\cdot,t)=:T\circ\tilde{f}(\cdot,t)\circ T^{-1}$, which forms a semigroup in the unit disc. This way, we are able to connect the radial to the chordal case. Considering further, the reparameterization $t\mapsto\frac{|B|}{\cos(\psi)}t$, and for $a:=\tan(\psi)$, we then write
\begin{equation}
\label{semigroupT}
g(z,t)=(h\circ T^{-1})^{-1}(e^{-(1-ia)t}h\circ T^{-1}(z))
\end{equation}
where $h\circ T^{-1}$, is an $(\text{Arccot}(a)-\frac{\pi}{2})$-spirallike function of $\mathbb{D}$. If $P_{\mathbb{D}}(z)=-\frac{\partial_t g(z,t)}{zg'(z,t)}$ is its infinitesimal generator, we then deduce the relation 
\begin{equation}
    \label{infinitesimalgenerators}
    P_{\mathbb{D}}(z)= (1-ia)B\dfrac{2i\text{Im}\beta P_{\mathbb{H}}(T^{-1}(z))}{(T^{-1}(z)-\beta)(T^{-1}(z)-\bar{\beta)}}=\dfrac{|B|(z-1)^2}{2i\text{Im}\beta\cos(\psi)z}P_{\mathbb{H}}(T^{-1}(z))
\end{equation}
for all $z\in\mathbb{D}$. 

The following proposition allows us to map the chordal case of section 4.1 to the radial case of section 3.3. Recall that we start with the upper half plane setting, thus we have the points $k_j\in\mathbb{R}$, the weights $b_j>0$ and we assume that the polynomial $P$ has a complex root $\beta\in\mathbb{H}$. Hence, in order to consider the radial analogue using the Möbius transform $Tz=\frac{z-\beta}{z-\overline{\beta}}$, we need to make a suitable choice for the weights corresponding to the configuration of the unit disc.
\begin{proposition}
    The function $\hat{g}(z,t):=g(e^{-iat}z,t)$, where $g$ is defined by $(\ref{semigroupT})$, is a radial Loewner flow driven by the function
    $$p(z,t):=\sum_{j=1}^{n}\dfrac{8b_j(\text{Im}\beta)^2|B|}{|\beta-k_j|^4\cos(\psi)}\dfrac{e^{iat}\zeta_j+z}{e^{iat}\zeta_j-z}$$
    where $\zeta_j:=T(k_j)=\frac{\beta-k_j}{\bar{\beta}-k_j}\in\partial\mathbb{D}$, for $1\le j\le n$, and $B$ is given by $(\ref{chordalcoefficients2})$.
\end{proposition}
\begin{proof}
    By $(\ref{semigroupT})$, the driving function is $p(z,t):=-\frac{\partial_t\hat{g}(z,t)}{z\hat{g}'(z,t)}=ia+P_{\mathbb{D}}(e^{-iat}z).$
    By the first equality of $(\ref{infinitesimalgenerators})$, following the proof of proposition $\ref{chordalspirallike}$, we rewrite $P_{\mathbb{D}}$ as the sum 
    $$P_{\mathbb{D}}(z)=\sum_{j=1}^{n}\frac{4b_j}{|\beta-k_j|^2}\frac{1}{T^{-1}(z)-k_j}=\sum_{j=1}^{n}\frac{4b_j}{|\beta-k_j|^2}\frac{1}{\bar{\beta}-k_j}\frac{z-1}{z-\zeta_j}$$
    and therefore, again by $(\ref{infinitesimalgenerators})$,
  \begin{align*}
      p(z,t)&=ia+(1-ia)\sum_{j=1}^{n}\frac{4b_j}{|\beta-k_j|^2}\frac{2i\text{Im}\beta B}{\bar{\beta}-k_j}\frac{z-e^{iat}}{z-\zeta_je^{iat}}\\
      &=1+ia-1+(1-ia)\sum_{j=1}^{n}\frac{4b_j}{|\beta-k_j|^2}\frac{2i\text{Im}\beta B}{\beta-k_j}\zeta_j\frac{z-e^{iat}}{z-\zeta_je^{iat}}.
  \end{align*}

   We, now, use relation $(\ref{chordalcoefficients2})$. By the same argument as before, we have that
   \begin{equation}\label{B}
       \frac{1}{2i\text{Im}\beta B}=\sum_{j=1}^{n}\frac{4b_j}{|\beta-k_j|^2(\beta-k_j)} \quad\text{or} \quad 1=\sum_{j=1}^{n}\frac{4b_j 2i\text{Im}\beta B}{|\beta-k_j|^2(\beta-k_j)}.
   \end{equation}
   By considering the real part in the right hand part of the preceding equation, we also deduce that 
     \begin{equation}\label{B'}
      1=\sum_{j=1}^{n}\frac{8b_j (\text{Im}\beta)^2 |B|}{|\beta-k_j|^4\cos(\psi)}.
   \end{equation}
   As a result, by $(\ref{B})$
   \begin{align*}
       p(z,t)&=1-(1-ia)\sum_{j=1}^{n}\frac{4b_j}{|\beta-k_j|^2}\frac{2i\text{Im}\beta B}{\beta-k_j}(1-\zeta_j\frac{z-e^{iat}}{z-\zeta_je^{iat}})\\
       &=1-(1-ia)\sum_{j=1}^{n}\frac{4b_j}{|\beta-k_j|^4}\frac{(2i\text{Im}\beta)^2 B(\bar{\beta}-k_j)}{2i\text{Im}\beta}\frac{z(1-\zeta_j)}{z-\zeta_je^{iat}}
   \end{align*}
 
   We finally observe that $\frac{\bar{\beta}-k_j}{2i\text{Im}\beta}(\zeta_j-1)=1$ and hence by $(\ref{B'})$ 
   \begin{align*}
       p(z,t)&=1+\sum_{j=1}^{n}\frac{8b_j}{|\beta-k_j|^4}\frac{(\text{Im}\beta)^2 |B|}{\cos(\psi)}\frac{2z}{\zeta_je^{iat}-z}\\
       &=\sum_{j=1}^{n}\dfrac{8b_j(\text{Im}\beta)^2|B|}{|\beta-k_j|^4\cos(\psi)}(1+\frac{2z}{\zeta_je^{iat}-z})
   \end{align*}
and the result follows.
\begin{figure}[ht]
	              \centering
	              \includegraphics[width=1
	              \linewidth]{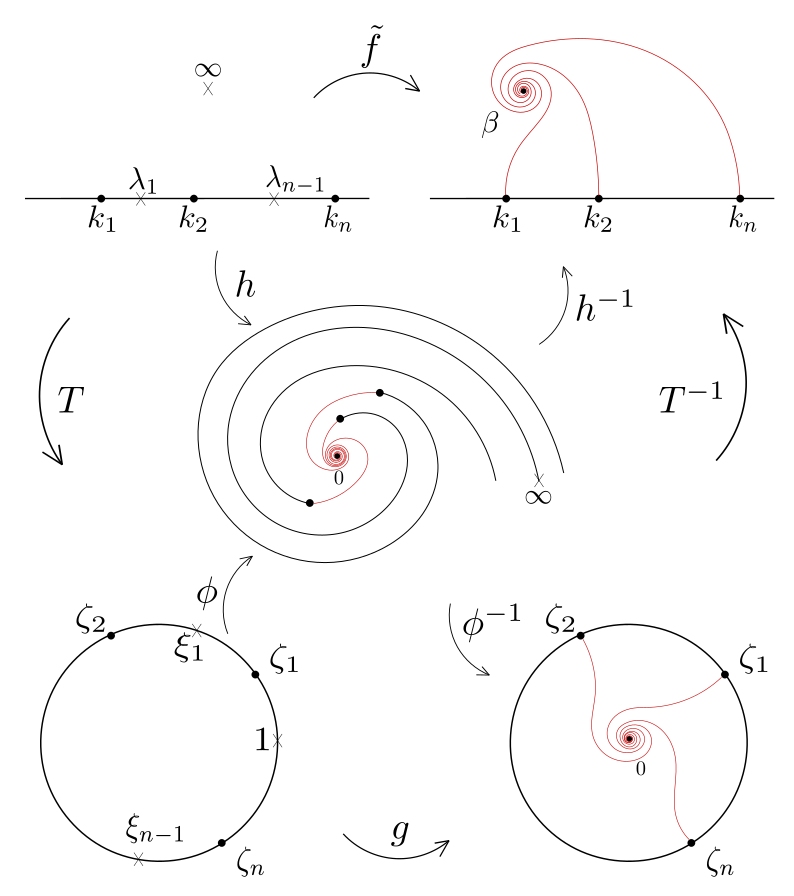}
	              \caption{Mapping chordal to radial.}
	              \label{fig:T}
	          \end{figure}	    

\end{proof}

\begin{rmk}
We can, therefore, connect the chordal flow to the radial flow by applying a Möbius transform that maps the attraction point $\beta\in\mathbb{H}$ to the origin and a suitable time reparameterization. This means that the points $k_j$ are mapped to some points $\zeta_j\in\partial\mathbb{D}$, the initial point masses. Following section 3.3, for these points and the parameters of the preceding proposition, there will exist points $\xi_j\in\partial\mathbb{D}$ (see lemma $\ref{parameters}$), playing the role of the preimages of infinity under $\phi$, where $\phi$ is given by theorem $\ref{Radialflow}$.

On the other hand, the preimages of infinity under $h$, given by $(\ref{h(z)})$, are the points $\lambda_1,\dots,\lambda_{n-1}$ (the real roots of P) and $\infty$. It is, thus, necessary that the points $\xi_j$ coincide with the images of these points under $T$, as we see in the figure above. To verify this, a short computation shows that

$$\sum_{j=1}^{n}\dfrac{8b_j(\text{Im}\beta)^2|B|}{|\beta-k_j|^4\cos(\psi)}\frac{z+\zeta_j\frac{1-ia}{1+ia}}{\zeta_j-z}=2i\text{Im}\beta\bar{B}\frac{P(T^{-1}z)}{(T^{-1}z-\beta)(T^{-1}z-\bar{\beta})}$$
where the first part of the equation determines the points $\xi_j$ as seen by lemma $\ref{parameters}$.
\end{rmk}
    
In view of the reparameterization $(\ref{chordalsemigroup0})$ and the time-change in PDE $(\ref{chordalsemigroup})$, we introduce the PDE in $\mathbb{H}\times[0,+\infty)$
\begin{equation}\label{generalchordalPDE}
\dfrac{\partial f}{\partial t}(z,t)=-f'(z,t)P_{\mathbb{H}}(z,t)
\end{equation}
with initial value $f(z,0)=z$, where $P_{\mathbb{H}}(\cdot,t)$ is an analytic function of the upper half plane given in the form
\begin{equation}
    \label{pH}
    P_{\mathbb{H}}(z,t)=e^tp(ze^{-t})-z
\end{equation}
for some analytic function $p\in H(\mathbb{H})$, such that the solution to (\ref{generalchordalPDE}) is a chordal Loewner flow. As in section 3.4, we assume a time-dependent chordal equation, however the dependence is "weak". The preceding assumption and the assumptions of the following proposition will give us a class of driving functions, for which the Loewner flows are given in terms of spirallike functions.
\begin{proposition}
        Consider the chordal PDE $(\ref{generalchordalPDE})$, with the assumptions of $(\ref{pH})$. Assume that there exists some $\beta\in\mathbb{H}$, such that $p(\beta)=0$, $p'(\beta)\neq0$ and that for all $z\in\mathbb{H}$
        $$\text{Im}\left(\dfrac{p(z)}{(z-\beta)(z-\bar{\beta)}}\right)<0.$$
        Then, the Loewner flow is given by the formula $(\ref{chordalsemigroup0})$, where $p'(\beta):=|B|e^{-i\psi}$ and $h\in H(\mathbb{H})$ maps the upper half plane onto some $(-\psi)$-spirallike domain with respect to the origin.
\end{proposition}
\begin{proof}
 Given some $p$, with the assumptions above, consider $h$ to be the solution to the ODE in $\mathbb{H}$,
 $$\dfrac{h'(z)}{h(z)}=\dfrac{p'(\beta)}{p(z)}.$$
so that $h(\beta)=0$. Then, by proposition \ref{spirallikeinH}, $h$ is $(-\psi)$-spirallike and define $f(z,t)$ by $(\ref{chordalsemigroup0})$. It is then straightforward to see that $f$ is a chordal Loewner flow satisfying $(\ref{generalchordalPDE})$.
\end{proof}

Similarly for the other cases, we see that by the same time reparameterization, we can write the flow $(\ref{chordalflow2})$ as $f(z,t)=h^{-1}(e^{\frac{t}{B_1}}h(e^t z)))$ and the flow $(\ref{chordalflow3})$ as $f(z,t)=h^{-1}(-t+h(e^tz))$. In both cases, we deduce that $\tilde{f}(z,t)=f(e^{-t}z,t)$ is a non-elliptic semigroup, with Denjoy-Wolff point $\rho_0\in\mathbb{R}$.

	\end{document}